\documentclass{amsart}
\usepackage[utf8]{inputenc}

\usepackage{a4wide}
\usepackage{mathabx}
\usepackage[english]{babel}
\usepackage{amsthm}
\usepackage{comment}
\usepackage{enumitem}

\usepackage{xcolor}
\usepackage[colorlinks,
    linkcolor={red!50!black},
    citecolor={blue!50!black},
    urlcolor={blue!80!black}]{hyperref}

\usepackage{tikz}
\usepackage[all]{xy}
\usepackage{amsmath}
\usepackage{amsthm}
\usepackage{amssymb}
\usepackage{amsfonts}
\usepackage{amsopn}
\usepackage{amsrefs}
\usepackage{hyperref}
\usepackage[color=green!30]{todonotes}

\usepackage{import}

\usetikzlibrary{calc,cd,intersections,decorations.markings,patterns,decorations.pathreplacing,matrix,arrows}

\def\Z{\mathbb{Z}}
\def\Q{\mathbb{Q}}
\def\R{\mathbb{R}}
\def\F{\mathbb{F}}

\def\C{\mathbb{C}}
\def\ee{\mathfrak{e}}
\def\ff{\mathfrak{f}}
\def\L#1{{#1}[t^{\pm 1}]}
\def\O#1{#1(t)}
\def\LF{\L{\F}}
\def\LR{\L{\R}}
\def\LC{\L{\C}}
\def\OF{\O{\F}}
\def\OR{\O{\R}}
\def\OC{\O{\C}}
\def\OO{\mathcal{O}}
\def\wt#1{\widetilde{#1}}

\def\ol#1{\overline{#1}}
\def\wh#1{\widehat{#1}}
\def\pairing{\operatorname{\lambda}}
\def\ds{\delta\sigma}
\def\hodgep{\mathcal{P}}
\def\hodgeq{\mathcal{Q}}
\def\basicR{\mathsf{R}}
\def\basicC{\mathsf{C}}
\def\basicF{\mathsf{F}}
\def\module{\mathcal{M}}
\def\makeithash#1{#1^\#}
\def\makeithashT#1{#1^{\#T}}
\def\eps{\epsilon}

\newcommand{\sign}{\operatorname{sign}}
\newcommand{\im}{\operatorname{im}}

\DeclareMathOperator{\ord}{ord}

\DeclareMathOperator{\sgn}{sign}

\DeclareMathOperator{\Hom}{Hom}

\DeclareMathOperator\iim{Im}
\DeclareMathOperator\re{Re}

\DeclareMathOperator{\Bl}{Bl}

\newcommand{\bsm}{\left(\begin{smallmatrix}}
\newcommand{\esm}{\end{smallmatrix}\right)}

\newtheorem{theorem}{Theorem}[section]

\newtheorem{corollary}[theorem]{Corollary}
\newtheorem{lemma}[theorem]{Lemma}
\newtheorem{proposition}[theorem]{Proposition}
\newtheorem*{thmintro}{Theorem}
\newtheorem*{propintro}{Proposition}

\theoremstyle{definition}
\newtheorem{definition}[theorem]{Definition}
\newtheorem*{defintro}{Definition}

\newtheorem{example}[theorem]{Example}

\theoremstyle{remark}
\newtheorem{remark}[theorem]{Remark}

\theoremstyle{claim}

\newtheorem*{claim*}{Claim}

\renewcommand{\doteq}{\stackrel{\bullet}{=}}

\numberwithin{equation}{section}

\title{Twisted Blanchfield pairings and twisted signatures I: Algebraic background}

\author{Maciej Borodzik}
\address{Institute of Mathematics, University of Warsaw, ul. Banacha 2, 02-097 Warsaw, Poland}
\email{mcboro@mimuw.edu.pl}
\author{Anthony Conway}
\address{Massachusetts Institute of Technology, Cambridge MA 02139}
\email{anthonyyconway@gmail.com}
\author{Wojciech Politarczyk}
\address{Institute of Mathematics, University of Warsaw, ul. Banacha 2, 02-097 Warsaw, Poland.}
\email{wpolitarczyk@mimuw.edu.pl}

\begin{document}
\begin{abstract}
  This is the first paper in a series of three devoted to studying twisted linking forms of knots and three-manifolds.
  Its function is to provide the algebraic foundations for the next two papers by describing how to define and calculate signature invariants associated to a linking form~$M\times M\to\mathbb{F}(t)/\mathbb{F}[t^{\pm1}]$ for $\F=\R,\C$, where $M$ is a torsion $\mathbb{F}[t^{\pm 1}]$-module.
  Along the way, we classify such linking forms up to isometry and Witt equivalence and study whether they can be represented by matrices.
\end{abstract}

\maketitle

\section{Introduction}

Let $M$ be a finitely generated $\LF$-torsion module, where~$\F=\R,\C$. 
Endow $\LF$ with the involution 
 $p:=\sum_i a_it^i \mapsto p^\#:= \sum_i \overline{a}_i t^{-i}$ where  $\overline{a}_i$ denotes complex conjugation, and
suppose that~$\pairing$ is a non-singular sesquilinear Hermitian pairing~$\pairing \colon M \times M \to \F(t)/\LF$.
Such a pair~$(M,\pairing)$ is usually referred to as a \emph{linking form}. 
One of the main goals of this paper is to associate a (non-trivial) piecewise constant function~$\sigma=\sigma_{(M,\pairing)} \colon S^1 \to \Z$ to each linking form.
The function~$\sigma$ is additive under direct sum of linking forms and vanishes if~$(M,\pairing)$ is \emph{metabolic}, i.e. if there is a submodule~$P \subset M$ with~$P=P^\perp$,  where $P^\perp=\lbrace x  \in M  \mid \lambda(x,y)=0 \text{ for all } y \in P \rbrace$.
An additional requirement is that~$\sigma_{(M,\pairing)}$ should be algorithmically computable once one knows the isometry type of~$(M,\pairing)$.
As we now describe, our motivation for studying this question comes from knot theory and the results of this paper are used in~\cite{BCP_Top,BCP_Compu,ConwayKimPolitarczyk,ConwayNagelFibered}  to study twisted signature invariants of knots and questions related to knot concordance.

\subsection{Signature and linking forms in knot theory}

We describe
one of the numerous definitions of the Levine-Tristram signature~$\sigma_K \colon S^1 \to \Z$ of a knot~$K$~\cite{LevineKnotCob,Tristram}; see~\cite{ConwaySurvey} for a survey.
The \emph{exterior} of a knot $K \subset S^3$ refers to the complement $X_K:=S^3 \setminus \nu (K)$ of a tubular neighbhorhood $\nu(K) \cong S^1 \times D^2$ of $K$.
For every knot $K$,  Alexander duality ensures that $H_1(X_K) \cong \Z$ and therefore the kernel of  the abelianisation homomorphism $\pi_1(X_K) \twoheadrightarrow  H_1(X_K)\cong \Z$ determines a  cover $X_K^\infty \to X_K$ with deck transformation group $\Z$.
The~$\F$-vector space~$H_1(X_K^\infty;\F)$
is a module over~$\F[\Z]=\F[t^{\pm 1}]$ and is endowed with a linking form
$$ \Bl(K) \colon H_1(X_K^\infty;\F) \times H_1(X_K^\infty;\F) \to \F(t)/\F[t^{\pm 1}].$$ 
This linking form, called the \emph{Blanchfield pairing} (after Blanchfield~\cite{Blanchfield}) is \emph{representable} meaning that there is a size $n$ Hermitian matrix~$A$ over~$\F[t^{\pm 1}]$ with~$\det (A) \neq 0$ such that~$\Bl(K)$ is isometric to 
\begin{align}
  \label{eq:RepresentableIntro}
  \pairing_{A} \colon  \F[t^{\pm 1}]^n /A^T\F[t^{\pm 1}]^n \times \F[t^{\pm 1}]^n/A^T \F[t^{\pm 1}]^n &\to \F(t)/\F[t^{\pm 1}] \\
 ([x],[y]) & \mapsto x^TA^{-1} y^\#. \nonumber
\end{align}
A matrix \(A(t)\) representing~\(\Bl(K)\) can be constructed with the aid of a Seifert matrix~\cite{KeartonBlanchfieldSeifert,FriedlPowell}.
The Levine-Tristram signature can then be defined as~$\sigma_K(\omega)=\sign(A(\omega))-\sign(A(1))$, where~$A(t)$ is \emph{any} matrix representing~$\Bl(K)$.
Summarising, in classical knot theory, the existence of Seifert matrices ensures that a (fairly computable) signature function can be associated to~$\Bl(K)$.
The same can be said whenever a linking form~$(M,\pairing)$ over~$\LF$ is representable and if a representing matrix is easy to obtain.

Unfortunately, some of the more involved knot invariants, known as \emph{twisted knot invariants} (we elaborate briefly on these in Subsection~\ref{sub:TwistedKnotIntro} below), lead to a theory where the resulting linking forms  admit no readily available presentation matrix.
This motivates our goal of defining computable signature invariants for arbitrary linking forms over~$\LF$ where, from now on~$\F$ is either~$\R$ or~$\C$.

\subsection{A classification result for linking forms}

Instead of immediately describing the outcome of our construction, we list the algebraic results we obtain along the way, as these might be of independent interest.

We start with a classification result of linking forms over~$\LF$.
This classification is obtained by showing that any linking form can be decomposed into certain \emph{basic linking forms}.
Basic linking forms are forms on cyclic modules of type~$\LF/\basicF_\xi^n$, where~$\xi$ is a non-zero complex number and~$\basicF_\xi$
is a \emph{basic polynomial}, which is roughly the minimal symmetric polynomial over~$\F$ vanishing at~$\xi$ 
(see Subsection~\ref{sec:classification-theorem}). 
These forms can be described explicitly and are denoted~$\mathfrak{e}(n,\xi,\pm 1,\F)$ and~$\mathfrak{f}(n,\xi,\F)$, 
according to whether or not~$\xi \in S^1 \subset \C$.

A particular case of one of our main results about linking forms is the following classification result.
\begin{thmintro}[Classification Theorem~\ref{thm:MainLinkingForm}]
  Every non-singular linking form over~$\LF$ is isometric to one of the form
  \begin{equation}\label{eq:splittingIntro}
    \bigoplus_{\substack{ n_i,\eps_i,\xi_i\\ i\in I}}\ee(n_i,\eps_i,\xi_i,\F)\oplus
    \bigoplus_{\substack{n_j,\xi_j\\ j\in J}}\ff(n_j,\xi_j,\F)
  \end{equation}
  for some finite sets of indices~$I$ and~$J$, where~$n_i,n_j > 0$ are integers,~$\xi_i\in\Xi^\F_\ee$ and~$\xi_j\in\Xi^\F_\ff$, and the sets~$\Xi^\F_\ee$,~$\Xi^\F_\ff$
  are defined in \eqref{eq:big_xi_def}.
  The decomposition is unique up to permuting summands.

\end{thmintro}

The fact that $\Q/\Z$-valued linking forms
admit such a decomposition has been known since the work of Wall and Kawauchi-Kojima~\cite{WallLinkingForm, KawauchiKojima}, while the result over~$R=\LR$ follows promptly from~\cite{milnor_isometries}, see also \cite[Section 4]{BorodzikFriedl2}. 
On the other hand, to the best of our knowledge, the result over~$R=\LC$ is new and involves difficulties not present in~\cite{BorodzikFriedl2}.
The fact that linking forms over $\LF$ can be classified is no doubt known to experts,  as can be gleaned from~\cite{LevineAlgebraic,
LevineMetabolicHyperbolic, knus_quadratic_1991,quebbemann_quadratic_1979}.
Our contribution in Theorem~\ref{thm:MainLinkingForm} is therefore not so much the existence of a classification as much as the description of the explicit linking forms it involves.

\subsection{Hodge numbers and signature jumps}

As a consequence of Classification Theorem~\ref{thm:MainLinkingForm}, we are able to define complete isometry invariants, and in particular, Witt equivalence invariants of non-singular linking forms over~$\LF$.

\begin{defintro}[Definition~\ref{def:hodge_number}]
  Given~$n >0,\xi \in\Xi^\F_\ee$ and~$\epsilon=\pm 1$, respectively,~$n>0$,~$\xi\in\Xi^\F_\ff$,
  we define the \emph{Hodge number}~$\mathcal{P}(n,\epsilon,\xi,\F)$, respectively~$\hodgeq(n,\xi,\F)$ of a non-singular linking form~$(M,\pairing)$ over~$\LF$ as the number of times~$\ee(n,\epsilon,\xi,\F)$, respectively~$\ff(n,\xi,\F)$ 
  enters the decomposition in~\eqref{eq:splittingIntro}.
\end{defintro}

The definition of these Hodge numbers is motivated by N\'emethi's work in singularity theory~\cite{Ne95};
the relation between \cite{Ne95} and the present paper is outlined in the survey paper \cite{BZ}.
The uniqueness of the decomposition in Classification Theorem~\ref{thm:MainLinkingForm} implies that 
the~$\hodgep(n,\epsilon,\xi,\F)$,~$\hodgeq(n,\xi,\F)$ are complete invariants of linking forms.
These Hodge numbers  lead to one of the main ingredients needed to define our twisted signature function.

\begin{defintro}[see Definition~\ref{def:sigjump}]
  \emph{The signature jump} of a linking form~$(M,\pairing)$ over~$\LC$ at~$\xi \in S^1$ is defined as the following integer:
  \begin{equation}
    \label{eq:JumpIntro}
    \ds_{(M,\pairing)}(\xi)=-\sum_{\substack{n \textrm{ odd}\\ \eps=\pm 1}} \eps \hodgep(n,\eps,\xi,\C).
  \end{equation}
\end{defintro}
The signature jump of a real linking form is defined analogously.
Before using signature jumps to construct the signature function, we describe how they can be used to classify linking forms up to Witt equivalence.
As we will recall more carefully in Section~\ref{sec:Witt},
the \emph{Witt group}~$W(\F(t),\LF)$ of linking forms over~$\LF$ roughly consists of the monoid of non-singular linking forms modulo the submonoid of metabolic  linking forms.
Two linking forms over~$\LF$ are \emph{Witt equivalent} if they agree in~$W(\F(t),\LF)$.
This Witt group is well understood~\cite{LitherlandCobordism, LevineMetabolicHyperbolic, RanickiHigh} and its description can be summarized by noting that
\begin{align*}
  W(\R(t),\LR) =\bigoplus_{\Xi^\R_{\ee+}} \Z, \ \ \ \ \ \  \ \ \ 
  W(\C(t),\LC) =\bigoplus_{\Xi^\C_{\ee}} \Z,
\end{align*}
where, $\Xi^\C_\ee=S^1$ and $\Xi^\R_{\ee+}=S^1_+=\lbrace z \in S^1 \ | \ \iim(z)>0 \rbrace$ denotes the upper half of the unit circle.
The following result (which combines
Theorems~\ref{thm:WittClassificationReal} and~\ref{thm:WittClassificationComplex}) uses our classification result to show that the signature jumps are complete invariants of non-singular linking forms over~$\LF$ up to Witt equivalence; for conciseness we write~$S^1_\R=\Xi^\R_{\ee+}$ and~$S^1_\C=\Xi^\C_\ee$.

\begin{thmintro}
  Any non-singular linking form over~$\LF$ is Witt equivalent to a finite direct sum 
 ~$$ \bigoplus_{\substack{ n_i \text{ odd, }\eps_i=\pm 1
      \\ \xi_i \in S^1_\F, \  i\in I 
    }}\ee(n_i,\eps_i,\xi_i,\F).$$
  The signature jumps are complete invariants of linking forms up to Witt equivalence over~$\LF$.
\end{thmintro}

At this point, the reader might wonder about our choice of terminology for the~$\delta \sigma_{(M,\pairing)}(\xi)$. 
To explain this, we return to representable forms.

\subsection{Representable linking forms}

Representability is a desirable property: in a nutshell, if a linking form~$\pairing \cong \pairing_A$ is representable (recall~\eqref{eq:RepresentableIntro}), then invariants of~$\pairing$ can be obtained from invariants of the matrix~$A$.
For instance, Proposition~\ref{prop:JumpIsJump} desribes how to obtain the signature jumps from a representing matrix and simultaneously justifies our terminology.
\begin{propintro}[Proposition~\ref{prop:JumpIsJump}]
  Let~$\xi\in S^1$. If a non-singular linking form~$(M,\pairing)$ over~$\LF$ is representable by a matrix~$A(t)$, then the following equation holds:
  \begin{equation*}
    \lim_{\theta\to 0^+} \sign A(e^{i \theta}\xi)-\lim_{\theta\to 0^-}\sign A(e^{i \theta}\xi)=2\ds_{(M,\pairing)}(\xi).
  \end{equation*}
\end{propintro}

Non-singular linking forms over~$\LR$ are always representable. Moreover, if multiplication by~$t\pm 1$ is an isomorphism, the form
is representable by a diagonalisable matrix~\cite[Proposition 4.1]{BorodzikFriedl2}. The situation over~$\LC$ is substantially different.
There are forms that are non-representable, see Subsection~\ref{sec:Representability}. Furthermore, forms that cannot be represented by
a diagonal matrix, occur much more often over the complex numbers.
Despite these surprises, using signature jumps,  we are able to prove the following characterisation of representable linking forms. It is stated as Corollary~\ref{cor:WittRepresentability},
but most of the technical work is done in Subsection~\ref{sec:rep2}.
\begin{propintro}[Corollary~\ref{cor:WittRepresentability}]
  Over~$\C[t^{\pm 1}]$, metabolic forms are representable, representability is invariant under Witt equivalence and, given a non-singular linking form~$(M,\pairing)$, the following are equivalent:
  \begin{enumerate}
  \item $(M,\pairing)$ is representable;
  \item $(M,\pairing)$ is Witt equivalent to a representable linking form;
  \item the \emph{total signature jump}~$\Sigma_{\xi \in S^1} \delta \sigma_{(M,\pairing)}(\xi)$ vanishes.
  \end{enumerate}
\end{propintro}

\begin{remark}
  The question of representability up to Witt equivalence can be understood via the long exact sequence in L-theory~\cite{Karoubi,Pardon,RanickiLocalization,RanickiExact}: it asks whether the map~$W(\F(t)) \to W(\F(t),\LF)$ is surjective.
  While this sequence is well understood for~$\LF$ (and more generally for Dedekind domains), to the best of our knowledge, the explicit equivalences for Corollary~\ref{cor:WittRepresentability} are new.
  In particular, the question of representability up to isometry (instead of Witt equivalence) does not appear to be addressed in the literature.
\end{remark}

Using signature jumps, we now describe the signature function and its properties.

\subsection{The signature function of a linking form}
\label{sub:SignatureFunctionIntro}

Our approach to signatures is based on signature jumps. 
Since signature functions in knot theory are locally constant, in order to define~$\sigma_{(M,\pairing)}(\omega)$, the idea is to sum the jumps at the discontinuities.

\begin{definition}
  \label{def:SignatureFunctionIntro}
  The \emph{signature function} of a non-singular linking form~$(M,\pairing)$ at~$\xi_1=e^{2\pi i \theta_1}$ is defined via the following formula:
  \[
    \sigma_{(M,\pairing)}(\xi_1)=\sum_{\tau\in(0,\theta_1)} 2\ds_{(M,\pairing)}(e^{2\pi i \tau})-\sum_{\substack{\eps=\pm 1\\n\textrm{ even}}} \eps \hodgep(n,\eps,\xi_1,\F)+\ds_{(M,\pairing)}(\xi_1)+\ds_{(M,\pairing)}(1).
  \]
\end{definition}

In other words, starting from a value at~$1$, we sum up the jumps encountered along the arc from~$1$ to~$\xi$, while appropriately taking into account the value of the signature jump at~$\xi$.

\begin{remark}
\label{rem:ValueAt1}
We comment briefly on the value of  signature function $\sigma_{(M,\lambda)}$ at $\xi_1=1$.
In classical knot theory,  the Levine-Tristram signature $\sigma_K$ of a knot $K$ satisfies $\sigma_K(1)=0$ and~$\sigma_K(e^{i\theta})=0$ as $e^{i\theta}$ approaches $1$.
One of the underyling reasons for these facts is that the Alexander module~$H_1(X_K^\infty;\R)$ does not contains any summand of the form $\R[t^{\pm 1}]/(t-1)^n$ in its primary decomposition.
Since our construction is meant to generalise $\sigma_K$,  it is natural for us to impose that~$\sigma_{(M,\lambda)}(1)$ vanish or at least that it vanish when $M$ contains no $(t-1)$-primary summand.
We choose the latter less stringent option mostly for cosmetic reasons,  as it ensures that the formulas in Propositions~\ref{prop:BasicProperties} and~\ref{prop:JumpIsJump} take particularly simple forms.
\end{remark}

The next result summarizes the properties of the signature function.
\begin{thmintro}
\label{thm:SignaturePropertiesIntro}
  Given a non-singular linking form~$(M,\pairing)$ over~$\LF$,
  the signature function~$\sigma_{(M,\pairing)} \colon S^1 \to \Z$ 
  satisfies the following properties:
  \begin{enumerate}
  \item The signature function is locally constant on~$S^1 \setminus \lbrace \xi \in S^1 \ | \ \Delta_M(\xi)=0 \rbrace$,  where~$\Delta_M$ denotes the order of the~$\LF$-module~$M$.
  \item If~$(M,\pairing)$ is a linking form over~$\LR$, then we have 
   ~$$\sigma_{(M,\pairing)}(\overline{\xi})=\sigma_{(M,\pairing)}(\xi).$$
  \item The signature function is additive: 
   ~$$\sigma_{(M_1,\pairing_1) \oplus (M_2,\pairing_2)}(\xi)=\sigma_{(M_1,\pairing_1)}(\xi)+\sigma_{(M_2,\pairing_2)}(\xi).$$
  \item The \emph{averaged signature function}
$$  \sigma^{av}_{(M,\pairing)}(\xi):=\frac12\left(\lim_{\theta\to 0^+}\sigma_{(M,\pairing)}(e^{i\theta}\xi)+\lim_{\theta\to 0^-}\sigma_{(M,\pairing)}(e^{i\theta}\xi)\right)
   ~$$
    is locally constant, additive and vanishes if~$(M,\pairing)$ is metabolic.
  \item If~$(M,\pairing)$ is represented by a matrix~$A(t)$, then the averaged signature function satisfies
   ~$$  \sigma^{av}_{(M,\pairing)}(\xi) =\sign^{av} A(\xi)-\sign^{av}A(1).$$
  \end{enumerate}
\end{thmintro}
The first two properties are proved in Proposition~\ref{prop:BasicProperties}. The third follows from the definitions (the Hodge numbers
are additive). The fourth property is Corollary~\ref{cor:metabolic_signature}. The last property is Proposition~\ref{prop:JumpIsJump}.

The properties listed in Theorem~\ref{thm:SignaturePropertiesIntro}, as well as the properties of the signature jumps listed above, are applied in our subsequent work in knot theory, as we now briefly describe.

\subsection{Applications to knot theory}
\label{sub:TwistedKnotIntro}

Over the past years, much interest in knot theory has revolved around so-called twisted knot invariants.
The theory of twisted knot invariants associates invariants to a knot~$K \subset S^3$ \emph{together with} a choice of a representation~$\rho$ of its knot group~$\pi_1(S^3 \setminus K)$.
Refraining from describing the vast literature on the subject (but referring to~\cite{FriedlVidussiSurvey} for a survey of twisted Alexander polynomials), we simply note that the Blanchfield pairing~$\Bl(K)$ generalises to a \emph{twisted Blanchfield pairing}~$\Bl_\rho(K)$~\cite{MillerPowell,Powell}.
The machinery developped in this paper has been applied in~\cite{BCP_Top} to study the properties of this linking form and to extract a twisted signature function~$\sigma_{K,\rho} \colon S^1 \to \Z$, and then applied in \cite{BCP_Compu,ConwayKimPolitarczyk,ConwayNagelFibered} to the study of problems related to concordance of algebraic knots.
Finally, it is worth mentioning that the satellite formula of~\cite{BCP_Top} (on which \cite{BCP_Compu,ConwayKimPolitarczyk} rely heavily) makes use of Theorem~\ref{thm:isoprojection} as a crucial ingredient.

\subsection{Previous results: signatures, linking forms and isometric structures}
\label{sub:ComparisonIntro}

Given~$\xi \in S^1$, Milnor~\cite{MilnorInfiniteCyclic} associated signature invariants~$\delta_{(M,\mu,t)}(\xi)$ to an isometric structure~$(M,\mu,t)$.\footnote{Here,~$(M,\mu,t)$ is an \emph{isometric structure} if~$M$ is a~$\F$-vector space,~$b \colon M \times M \to \F$ is a skew-Hermitian form, and~$t \colon M \to M$ is an isometry of~$b$.}
In the case where~$(M,\mu,t)=(H_1(X_K^\infty),\mu_K,t)$, with~$\mu_K$ the Milnor pairing of a knot~$K$~\cite[Sections 4-5]{MilnorInfiniteCyclic}, Matumoto related these \emph{Milnor signatures} to the Levine-Tristram signature~$\sigma_K(\omega)$~\cite{Matumoto}. 
On the other hand, as we mentioned in Section~\ref{sub:SignatureFunctionIntro}, Kearton~\cite{KeartonSignature} and Levine~\cite{LevineMetabolicHyperbolic} associated signature invariants to a linking form~$(M,\pairing)$ and related them to the Milnor signatures~$\delta_{(M,\mu,t)}(\xi)$ and to~$\sigma_K(\xi)$; see also~\cite[Appendix A]{LitherlandCobordism} and~\cite[Section 7-8]{ConwayNagelFibered}.
While we do not describe these results in detail, conceptually, the underlying statements resemble Proposition~\ref{prop:JumpIsJump}.

Crucially however, when relating these signatures to the Levine-Tristram signature, Matumoto and Levine both rely on the presence of Seifert matrices.
As a consequence, their proofs are not suitable to systematically study twisted signatures from the point of view of linking forms.
Summarising, this paper provides the algebraic framework needed to study signature type invariants in knot theory, when the algebraic assumptions from ``classical knot theory" are not available.

\subsection{Organisation of the paper}
This paper is organized as follows. In Section~\ref{sec:LinkingFormClassification}, we state and prove our classification result for~$\LF$-linking forms. In Section~\ref{sec:FurtherProperties}, we apply these methods to linking forms over local rings and to the representability of linking forms. In Section~\ref{sec:Witt}, we study~$\LF$-linking forms up to Witt equivalence. In Section~\ref{sec:Signatures}, we define the signature jumps and the signature function of a~$\LF$-linking form. 

\subsubsection*{Acknowledgments}

We are indebted to Andrew Ranicki for his guidance and invaluable help while we were writing the paper.
We have benefited from discussions with several people, including Jae Choon Cha, Stefan Friedl, Patrick Orson and Mark Powell.
We also wish to thank the University of Geneva and the University of Warsaw at which part of this work was conducted.

The first author is supported by the National Science Center grant 2019/B/35/ST1/01120.
The second author thanks Durham University for its hospitality and was supported by an early Postdoc.Mobility fellowship funded by the Swiss FNS.
The third author is supported by the National Science Center grant 2016/20/S/ST1/00369.

\subsubsection*{Conventions}
If~$R$ is a commutative ring and~$f,g\in R$, we write~$f\doteq g$ if there exists a unit~$u\in R$ such that~$f=ug$.
For a ring~$R$ with involution, we denote this involution by~$x\mapsto x^\#$; the symbol~$\ol{x}$ is reserved for the complex conjugation.
Given a commutative ring with involution $R$,  we endow the ring of Laurent polynomials $R[t^{\pm 1}]$ with  the involution obtained by composing the involution on~$R$ with the involution obtained by extending  the map~$t \mapsto t^{-1}$.
In particular, for~$\LC$ the involution~$(-)^\#$ is the composition of the complex conjugation with the map~$t\mapsto t^{-1}$.
For example, if~$p(t)=t-i$, then~$p^\#(i)=0$, but~$\ol{p}(i)=2i$.
Given an~$R$-module~$M$, we denote by~$M^\#$ the~$R$-module that has the same underlying additive group as~$M$, but for which the action by~$R$ on~$M$ is precomposed with the involution on~$R$.
For a matrix~$A$ over~$R$, we write~$(\makeithash{A})^T$ for the transpose followed by the conjugation.

\section{Classification of split linking forms up to isometry}\label{sec:LinkingFormClassification}

This section is organized as follows.
In Subsection~\ref{sub:Terminology}, we introduce some basic terminology on linking forms.
Subsection~\ref{sec:basic_linking_forms} describes the building blocks for split linking forms.
In Subsection~\ref{sec:classification-theorem}, we state the Classification Theorem~\ref{thm:MainLinkingForm} for linking forms over~$\F[t^{\pm 1}]$ up to isometry.
The remainder of Section~\ref{sec:LinkingFormClassification} is devoted to proving Classification Theorem~\ref{thm:MainLinkingForm}.

\subsection{Linking forms}
\label{sub:Terminology}

In this subsection, we first introduce some terminology on linking forms over an integral domain~$R$ with involution.
We then restrict ourselves to the case where~$R=\LF$ where the field~$\F$ is either~$\R$ or~$\C$.  References include~\cite{LitherlandCobordism, RanickiExact, RanickiLocalization}.
\medbreak

Let~$R$ be an integral domain with a (possibly trivial) involution~$x \mapsto x^\#$.
Let~$Q$ be the field of fractions of~$R$ which inherits an involution from~$R$.
A \emph{linking pairing} consists of a finitely generated torsion~$R$-module~$M$ and an~$R$-linear map
\[\pairing^{\bullet} \colon M \to \operatorname{Hom}_R(M,Q/R)^\#.\]
In practice, we shall often think of a linking pairing as a pair~$(M,\pairing)$ with $\pairing \colon M \times M \to Q/R$ a sesquilinear pairing, meaning that $\lambda(px,qy)=p\overline{q}\lambda(x,y)$ for all $p,q\in R$ and $x,y \in M$.
Such a linking pairing is called \emph{Hermitian} if $\lambda(y,x)=\lambda(x,y)^\#$ for all $x,y\in M$. 
\begin{definition}
  \label{def:LinkingForm}
  A Hermitian linking pairing~$(M,\pairing)$ is called a \emph{linking form}.
\end{definition}

A linking form~$(M,\pairing)$ is \emph{non-degenerate} if~$\pairing^{\bullet}_M$ is injective and \emph{non-singular} if~$\pairing^{\bullet}_M$ is an isomorphism.
In the literature, linking forms are usually assumed to be either non-degenerate or non-singular.
Note that if the ring~$R$ is a PID, then non-degenerate linking forms over a torsion~$R$-module are in fact non-singular~\cite[Lemma 3.24]{Hillman}.
A \emph{morphism} of linking forms~$(M,\pairing_M)$ and~$(N,\pairing_{M'})$ consists of an~$R$-linear homomorphism~$\psi \colon M \to N$ which intertwines~$\pairing_M$ and~$\pairing_N$, i.e.\ such that~$\pairing_N(\psi(x),\psi(y))=\pairing_M(x,y)$.
A morphism~$\psi$ from~$(M,\pairing_M)$ to~$(N,\pairing_N)$ is an \emph{isometry} if~$\psi \colon M \to N$ is an isomorphism.

We will study linking forms over the ring of Laurent polynomials~$\F[t^{\pm 1}]$ where $\F=\R,\C$.  The involution on $\LF$ maps $p:=\sum_i a_it^i$ to $p^\#:=\sum_i \overline{a_i}t^{-i}$ where $\overline{a_i}$ denotes the complex conjugate of $a_i \in \C$.
This involution can be extended to an involution on the field of fractions $\F(t)$ of $\LF$ which then descends to an involution on $\F(t)/\LF$.
Finally, 
a Laurent polynomial~$p\in\F[t^{\pm 1}]$ is called \emph{symmetric} if~$p=p^\#$.

\subsection{$\xi$-positive polynomials}\label{sec:positive}
In order to deal effectively with linking forms on modules of the type~$\LC/(t-\xi)^k$, where~$\xi\in S^1$ and~$k$ odd, we introduce some terminology.

\begin{definition}\label{def:positive}
  Given~$\xi\in S^1$, a complex polynomial~$r(t)$ is called \emph{$\xi$-positive} if~$(t^{-1}-\ol{\xi})r(t)$ is a complex symmetric polynomial and the function
 ~$t\mapsto (t^{-1}-\ol{\xi})r(t)$ changes sign from positive to negative as~$t$ crosses the value~$\xi$ going counterclockwise on the unit circle.
\end{definition}

While one could work with explicit~$\xi$-positive polynomials, it is often useful to use the abstract notion instead.
To build up some intuition, we provide some concrete examples of~$\xi$-positive polynomials:

\begin{example}
  \label{ex:XiPositive}
  If~$\xi\neq\pm 1$ then~$r(t)=1-\xi t$ is~$\xi$-positive if~$\iim(\xi)>0$ and~$-(1-\xi t)$ is~$\xi$-positive if~$\iim(\xi)<0$.
  Indeed, if~$\iim(\xi)>0$, then~$(t^{-1}-\ol{\xi})r(t)=t-2\re\xi +t^{-1}$
  and so~$(e^{-i\theta}-\xi)r(e^{i\theta})<0$ if~$\re(\theta)<\re(\xi)$, that is if~$\theta<\theta_0$.

  One readily checks that~$-i(t+1)$ is~$\xi$-positive for~$\xi=1$ and~$i(t-1)$ is~$\xi$-positive for~$\xi=-1$.
\end{example}

\begin{remark}\label{rem:positive}
  The involution~$r(t)\mapsto r(t)^\#$ can be extended to the case where~$r(t)$ is a complex analytic function near~$\xi$: if~$r(t)=\sum a_j(t-\xi)^j$ near~$\xi$, then one sets~$r(t)^\#=\sum \ol{a}_j(t^{-1}-\ol{\xi})^j$.
  Definition~\ref{def:positive} can also be extended to the case where~$r(t)$ is complex analytic.
  Namely,~$r(t)$ is called~$\xi$-positive if~$(t^{-1}-\ol{\xi})r(t)$ is symmetric and changes sign from positive to negative when crossing the value~$\xi$, as in Definition~\ref{def:positive}.
  This extension is used in Section~\ref{sub:LocalizationDiagonalization}.
  
Similarly extending the definition from the polynomial setting,  we call an analytic function~$r(t)$ \emph{symmetric} if $r(t)=r(t)^\#$.
  Note that by the Identity Principle from complex analysis,~$r(t)$ is symmetric if and only if it takes real values on the unit circle.

\end{remark}
For later use, we collect the following result concerning~$\xi$-positivity.

\begin{lemma}\label{lem:residue}
Let $\xi \in S^1$.  If $r(t)$ is an analytic function
such that $r_0:=r(\xi)\neq 0$ and $r(t)(t^{-1}-\ol{\xi})$ is symmetric,  then $\re(\ol{\xi}r_0)=0$
  Moreover, $r(t)$ is $\xi$-positive if and only if $\iim(\ol{\xi}r_0)<0$.
  In particular, if $r$ and~$g$ are Laurent polynomials with~$r(\xi),g(\xi)\neq 0$, then $r$ is $\xi$-positive if
  and only if $rgg^\#$ is.
\end{lemma}
\begin{proof}
  Write~$r(t)$ as the power series~$r(t)=r_0+r_1(t-\xi)+\dots$.
  We have that
  \begin{equation}\label{eq:olxi}
    (t^{-1}-\ol{\xi})=-\ol{\xi}(t-\xi)t^{-1}=-\ol{\xi}^2(t-\xi)+\ol{\xi}^3(t-\xi)^2-\ol{\xi}^4(t-\xi)^3+\dots,
  \end{equation}
  where the second equality was obtained by writing the Taylor expansion 
  \[\frac{1}{t}=\frac{1}{\xi+(t-\xi)}=\frac{\ol{\xi}}{1+\ol{\xi}(t-\xi)}=\ol{\xi}-\ol{\xi}^2(t-\xi)+\dots.\]
  Using \eqref{eq:olxi}, we obtain
  \begin{equation}\label{eq:expandr}
    r(t)(t^{-1}-\ol{\xi})=-r_0\ol{\xi}^2(t-\xi)+\dots.
  \end{equation}
  By \eqref{eq:olxi} again we get
  \begin{multline*}
    (r(t)(t^{-1}-\ol{\xi}))^\#=r(t)^{\#}(-\ol{\xi}^2(t-\xi)+\ol{\xi}^3(t-\xi)^2-\dots)=\\
    -\ol{r_0}\xi^2(-\ol{\xi}^2(t-\xi)+\dots)+\dots=\ol{r_0}(t-\xi)+\dots
  \end{multline*}
  The condition that~$r(t)(t^{-1}-\ol{\xi})$ be symmetric implies that~$-r_0\ol{\xi}^2=\ol{r_0}$, that is~$\re (r_0\ol{\xi})=0$.

  For~$\varepsilon>0$ sufficiently small and real, set~$t_0=\xi+i\xi\varepsilon$.
  The point~$t_0$ lies close to the circle~$S^1$, and it is in the counterclockwise direction with respect to~$\xi$.
  The function $t\mapsto r(t)(t^{-1}-\ol{\xi})$ changes sign as $t$ moves through $\xi$ on the unit circle. It changes
  sign from positive to negative, if and only if the real part of~$r(t_0)(t_0^{-1}-\xi)$ is negative for $\varepsilon$ sufficiently small; if $r(t_0)(t_0^{-1}-\xi)$ has positive real part, then $t\mapsto r(t)(t^{-1}-\ol{\xi})$ is positive for $\varepsilon$ sufficiently small.
  That is to say, the sign of the real part of $r(t_0)(t_0^{-1}-\xi)$ determines $\xi$-positivity of $r$.

  We use now~\eqref{eq:expandr}. As~$r_0\neq 0$, we see that the real part of $r(t_0)(t_0^{-1}-\xi)$ is negative
  if and only if~$-i\xi\varepsilon r_0\ol{\xi}^2$ has a negative real part, equivalenlty~$r_0\ol{\xi}$ must have a negative imaginary part.

  \smallskip
  Assume that $r$ and $g$ are Laurent polynomials, $r(t)(t^{-1}-\ol{\xi})$ is symmetric, and $r(\xi),g(\xi)\neq 0$. 
  Set $\wt{r}=rgg^\#$. Then $\wt{r}(\xi)$ differs from $r(\xi)$ by a real positive factor. Hence, $r$ is $\xi$-positive
  if and only if $\wt{r}$ is.
\end{proof}

\subsection{Basic linking forms}\label{sec:basic_linking_forms}

Following the approach of \cite{LevineAlgebraic} and, independently, of~\cite{Ne95}, 
we first describe simple ``basic'' linking forms that will be the building blocks of all linking forms over~$\LF$. 
\medbreak

Before discussing the relevant linking forms, let us introduce the following notation.
Define
\begin{equation}\label{eq:big_xi_def}
  \begin{split}
    \Xi^\R_\ee&:=\{\xi\in\C \ | \ \iim (\xi)\ge 0,\ |\xi|=1\},\\
    \Xi^\R_\ff&:=\{\xi\in\C\ | \iim(\xi)\ge 0,\ |\xi|\in(0,1)\}\cup\{-1,1\},\\
    \Xi^\C_\ee&:=\{\xi\in\C \ | \ |\xi|=1\},\\
    \Xi^\C_\ff&:=\{\xi\in\C\ | \ |\xi|\in(0,1)\},\\
    \Xi^\R&=\Xi^\R_\ee\cup\Xi^\R_\ff,\ \ \Xi^\C=\Xi^\C_\ee\cup\Xi^\C_\ff.
  \end{split}
\end{equation}
These sets are illustrated in Figure~\ref{fig:Xi_sets}.
We also introduce
\begin{equation}\label{eq:big_xi_def_2} 
  \begin{split}
    \Xi^\R_{\ee+}&=\{\xi\in\C \ | \ \iim (\xi)> 0,\ |\xi|=1\}=\Xi^\R_{\ee}\setminus\Xi^\R_{\ff}.\\
    \Xi^\R_{\ff+}&=\{\xi\in\C \ | \ \iim (\xi)\ge 0,\ |\xi|<1\}=\Xi^\R_{\ff}\setminus\Xi^\R_{\ee}.
  \end{split}
\end{equation}
\begin{figure}[!htbp]
  \begin{tikzpicture}
    \begin{scope}[xshift=-3cm]
      \draw[blue, ultra thick] (0.5,0) arc [start angle=0, delta angle=180,radius=0.5];
      \draw[dashed](0,0) circle (0.5);
      \draw[->] (-1,0) -- (1,0);
      \draw[->] (0,-1) -- (0,1);
      \draw(0,-1.2) node[scale=0.8] {$\Xi^\R_\ee$};
    \end{scope}
    \begin{scope}[xshift=-0.5cm]
      \fill[blue!30] (-0.5,0) -- (0.5,0) arc[start angle=0,delta angle=180,radius=0.5];
      \draw[dashed](0,0) circle (0.5);
      \draw[->] (-1,0) -- (1,0);
      \draw[->] (0,-1) -- (0,1);
      \fill[blue] (0.5,0) circle (0.08);
      \fill[blue] (-0.5,0) circle (0.08);
      \draw(0,-1.2) node[scale=0.8] {$\Xi^\R_\ff$};
    \end{scope}
    \begin{scope}[xshift=2cm]
      \draw[ultra thick, blue](0,0) circle (0.5);
      \draw[->] (-1,0) -- (1,0);
      \draw[->] (0,-1) -- (0,1);
      \draw(0,-1.2) node[scale=0.8] {$\Xi^\C_\ee$};
    \end{scope}
    \begin{scope}[xshift=4.5cm]
      \fill[blue!30] (0,0) circle (0.5);
      \draw[dashed](0,0) circle (0.5);
      \draw[->] (-1,0) -- (1,0);
      \draw[->] (0,-1) -- (0,1);
      \draw(0,-1.2) node[scale=0.8] {$\Xi^\C_\ff$};
    \end{scope}
  \end{tikzpicture}
  \caption{The sets $\Xi^\R_\ee$, $\Xi^\R_\ff$, $\Xi^\C_\ee$ and $\Xi^\C_\ff$.}\label{fig:Xi_sets}
\end{figure}

For~$\xi\in\Xi^\R$, we refer to the following real polynomials as the (real) \emph{basic polynomials}:
\[
  \basicR_\xi(t):=\begin{cases}
    (t-\xi), & \text{for } \xi=\pm1,\\
    (t-\xi)(1-\ol{\xi}t^{-1}), & \text{for } |\xi|=1,\;\xi\neq\pm 1,\\
    (t-\xi)(t-\ol{\xi})(1-\ol{\xi}^{-1}t^{-1})(1-t^{-1}\xi^{-1}), & \text{for } |\xi|<1,\;\xi\notin\R,\\
    (t-\xi)(1-\xi^{-1}t^{-1}), & \text{for }  \xi\in (-1,1).
  \end{cases}
\]
The purpose of introducing basic polynomials is that any symmetric polynomial~$f$ can be decomposed as~$p_1^{n_1} \cdots p_{r}^{n_r}$, where the~$p_i$ are basic.
\begin{definition}
  A (real) \emph{basic module} is a module of the form
  \[\module(n,\xi,\R):=\LR/\basicR_\xi(t)^n,\]
  where~$n$ is a positive integer and~$\xi\in\Xi^\R$.
\end{definition}
Similarly, in the complex case, we define the \emph{basic polynomials} to be
\[\basicC_{\xi}(t) =
  \begin{cases}
    (t-\xi), & \text{if } \xi\in\Xi^\C_\ee,\\
    (t-\xi)(t^{-1}-\xi), & \text{if } \xi\in\Xi^\C_\ff.
  \end{cases}\]
By analogy with the real case, we define:
\begin{definition}
  A (complex) \emph{basic module} 
  is~$\module(n,\xi,\C):=\LC/\basicC_{\xi}(t)^n$, for some \(n>0\)
  and~$\xi\in\Xi^\C$.
\end{definition}

\begin{remark}
  We will use a uniform notation~$\basicF_\xi$ for a polynomial~$\basicR_\xi$ or~$\basicC_\xi$, depending on~$\F=\R$ or~$\F=\C$.
  The polynomials
 ~$\basicF_\xi(t)$ are characterized as the minimal degree polynomials over~$\LF$ that vanish at~$\xi$ and satisfy~$\basicF_\xi(t)\doteq\basicF_\xi(t)^\#$. 
\end{remark}

We introduce the basic linking forms; we start with the real case.

\begin{example}
  \label{BasicPairingXiS1Real}
  Fix a complex number~$\xi \in\Xi^\R_\ee$,~$\xi\notin\{-1,1\}$, a positive integer~$n$, and~$\eps= \pm 1$.
  Consider the basic linking form~$\ee(n,\eps,\xi,\R)$ defined on
$\module(n,\xi,\R)$ by the formula
  \begin{equation}\label{eq:e_n_k_form_real}
    \begin{split}
      \ee(n,\eps,\xi,\R) \colon \module(n,\xi,\R)\times\module(n,\xi,\R) & \to\OR/\LR, \\
      (x,y) &\mapsto \frac{\eps x y^{\#}}{\basicR_\xi(t)^{n}}. 
    \end{split}
  \end{equation}
This pairing is Hermitian since~$\basicR_\xi(t)$ is symmetric.
  Moreover, \(\ee(n,\eps,\xi,\R)\) is non-singular.
\end{example}

Next, we consider the case where~$\xi$ does not belong to~$S^1$.

\begin{example}
  \label{ex:RealBasicXiNotReal}
  Fix a complex number~$\xi\in\Xi^\R_{\ff+}$, and a non-negative integer~\(n\), and consider the linking form
 ~$\ff(n,\xi,\R)$:
  \begin{equation}\label{eq:f_n_k_form_real_weak}
    \begin{split}
      \module(n,\xi,\R) \times \module(n,\xi,\R) &\to \OR/\LR, \\
      (x,y) &\mapsto \frac{x y^{\#}}{\basicR_\xi(t)^{n}}.
    \end{split}
  \end{equation}
  The pairing is Hermitian and non-singular.
\end{example}

The case \(\xi=\pm 1\) is treated separately.

\begin{example}
  \label{BasicPairingXipm1Real}
  Let~$\xi = \pm 1$ and~$n>0$.
  \begin{itemize}
  \item Suppose~$n$ is even.
    Define the real linking form~$\ee(n,\xi,\eps,\R)$ as
    \begin{equation}\label{eq:e_n_k_form_real_xi_one}
      \begin{split}
        \module(n,\xi,\R) \times \module(n,\xi,\R) &\to\OR /\LR,\\
        (x,y) &\mapsto \frac{\eps x y^{\#}}{\basicR_\xi(t)^{n/2}\basicR_\xi(t^{-1})^{n/2}}. 
      \end{split}
    \end{equation}
    This pairing is Hermitian since~$\basicR_\xi(t)\basicR_\xi(t^{-1})$ is symmetric.
    It is also non-singular.  ~$\LR/ \basicR_\xi(t)^n$ and~$\LR/ \basicR_\xi(t)^{n/2}\basicR_\xi(t^{-1})^{n/2}$ are isomorphic.
  \item Suppose~$n$ is odd. Define \(\ff(n,\xi,\R)\) as
    \begin{align*}
      \module(n,\xi,\R)^{2} \times \module(n,\xi,\R)^{2} & \to \OR / \LR, \\
      ((x_{1},x_{2}),(y_{1},y_{2})) &\mapsto \frac{x_{1} y_{2}^{\#}}{(t-\xi)^{n}} + \frac{x_{2} y_{1}^{\#}}{(t^{-1}-\xi)^{n}}.
    \end{align*}
  \end{itemize}
\end{example}

In the complex case, elementary pairings are defined in a similar fashion, but there are only two cases to consider depending on the norm of \(\xi\).

\begin{example}
  \label{BasicPairingXiS1Complex}
  Fix a complex number~$\xi\in S^1=\Xi^\C_\ee$, a positive integer~$n$.
  Define the linking form~$\ee(n,\eps,\xi,\C)$ as
  \begin{align}
    \module(n,\xi,\C) \times \module(n,\xi,\C) &\to \OC/\LC, \nonumber\\
    (x,y) &\mapsto \frac{\eps x y^{\#}}{\basicC_{\xi}(t)^{n/2}\basicC_{\bar{\xi}}(t^{-1})^{n/2}}, &\text{if~$n$ is even}, \label{eq:e_n_k_form_complex_even}\\
    (x,y) &\mapsto \frac{\eps r(t) x y^{\#}}{\basicC_{\xi}(t)^{\frac{n+1}{2}}\basicC_{\bar{\xi}}(t^{-1})^{\frac{n-1}{2}}}, &\text{if~$n$ is odd}, \label{eq:e_n_k_form_complex_odd}
  \end{align}
  where the polynomial~$r(t)$ appearing in the numerator of \eqref{eq:e_n_k_form_complex_odd} is linear and~$\xi$-positive (note that the existence of such polynomials follows from Example~\ref{ex:XiPositive}).
\end{example}

It is shown in Theorem~\ref{thm:elementary_classification} that the isometry class of~$\ee(n,\eps,\xi,\C)$ does not depend on the choice of the linear~$\xi$-positive polynomial~$r(t)$. 
\begin{lemma}\ \label{lem:not_iso_ee_C}
  \begin{itemize}
  \item[(a)] The pairings~$\ee(n,\eps,\xi,\C)$ are Hermitian;
  \item[(b)] The pairings~$\ee(n,\eps,\xi,\C)$ are non-singular;
  \item[(c)] The pairings~$\ee(n,+1,\xi,\C)$ and~$\ee(n,-1,\xi,\C)$ are not isomorphic;
  \end{itemize}
\end{lemma}
\begin{proof}[Sketch of proof]
  For point (a), the case~$n$ even is obvious while the case~$n$ odd follows from the definition of~$\xi$-positivity; for this, note that~$C_{\ol{\xi}}(t^{-1})=C_\xi(t)^\#$.
  Following the same steps as in the real case, we easily show that these pairings are non-singular.

  As for (c), note that an automorphism of~$\module(n,\xi,\C)$ replaces~$r$ by~$rgg^\#$ for some~$g\in\LC$. We conclude by Lemma~\ref{lem:residue}.
\end{proof}

As in the real case, we conclude with the case where~$\xi$ does not belong to~$S^1$.
\begin{example}
  Let~$\xi\in\Xi^\C_\ff$. Fix a non-negative integer \(n\) and consider the linking form
 ~$\ff(n,\xi,\C)$:
  \begin{equation}\label{eq:f_n_k_form_complex_weak}
    \begin{split}
      \module(n,\xi,\C) \times \module(n,\xi,\C) & \to \OC/\LC, \nonumber \\
      (x,y) &\mapsto \frac{x y^{\#}}{\basicC_{\xi}(t)^{n}}.
    \end{split}
  \end{equation}
  The pairing is Hermitian and non-singular.
\end{example}


\subsection{The classification theorem}
\label{sec:classification-theorem}

The goal of this subsection is to state the classification theorem for linking forms over~$\LF$ up to isometry, for \(\F = \R,\C\).
As a corollary of the classification theorem, we introduce complete numerical invariants of linking forms, the \emph{Hodge numbers}.\

\begin{theorem}[Classification Theorem]
  \label{thm:MainLinkingForm}
  Suppose~$\F$ is either~$\R$ or~$\C$.
  Every linking form~$(M,\pairing)$ over~$\LF$ is isometric a finite direct sum 
  \begin{equation}\label{eq:splitting}
    \bigoplus_{\substack{ n_i,\eps_i,\xi_i\\ i\in I}}\ee(n_i,\eps_i,\xi_i,\F)\oplus
    \bigoplus_{\substack{n_{j},\xi_j\\ j\in J}}\ff(n_{j},\xi_j,\F).
  \end{equation}  
  Here~$I$ and~$J$ are finite sets of indices,~$\xi_i\in\Xi^\F_\ee$,
 ~$\xi_j\in\Xi^\F_\ff$ and $n_i,n_j>0$ are integers.
  The above decomposition is unique up to permutating the factors.
\end{theorem}
\begin{remark}
  \label{rem:OrderAndClassification}
  In the decomposition~\eqref{eq:splitting}, there can only be finitely many non-trivial summands.
  Note furthermore that since we are working over a PID, the order of~$M$ is equal to the product of the basic polynomials which enter the decomposition.
\end{remark}
Before we prove Classification Theorem~\ref{thm:MainLinkingForm},
we state some of its consequences.
Classification Theorem~\ref{thm:MainLinkingForm} motivates the following definition.
\begin{definition}\label{def:hodge_number}
  Let~$n$ be a positive integer, and let~$\eps$ be~$\pm 1$.
  If ~$\xi\in\Xi^\F_\ee$ (resp.~$\xi\in\Xi^\F_\ff$), the \emph{Hodge number}~$\hodgep(n,\eps,\xi)$ (resp.~$\hodgeq(n,\xi)$) of a non-singular linking form~$(M,\pairing)$ is the number of times~$\ee(n,\eps,\xi,\F)$ (resp.~$\ff(n,\xi,\F)$) enters the sum~\eqref{eq:splitting}. 
\end{definition}
From~Theorem~\ref{thm:MainLinkingForm} deduce the following result.
\begin{corollary}\label{prop:hodge_are_inv}
  The Hodge numbers are complete invariants of non-singular linking forms.
\end{corollary}
The definition of the Hodge numbers is essentially the same as in \cite{Ne95} (in that paper, one takes~$\F=\R$) with one exception, namely~$\xi=1$.
When $\xi=1$,  N\'emethi's definition is different; it might happen that~$\hodgep(n,\eps,1,\R)\neq 0$ even if~$n$ is odd. 
The calculation of N\'emethi's Hodge numbers for~$\xi=1$ is usually more involved. 

Before we begin the proof of Classification Theorem~\ref{thm:MainLinkingForm},
we mention that several statements in the literature  resemble Classification Theorem~\ref{thm:MainLinkingForm}. While this particular statement seems to be new, it can be deduced from existing results in a more or less direct way. 
By this we mean that substantial steps of our proof have already appeared in many places.
On the one hand, Milnor's methods~\cite{milnor_isometries} can be applied with small changes
to give a proof; likewise reading carefully~\cite[Theorem 1.3]{LevineMetabolicHyperbolic} provides another approach.
On the other hand, analogous classification theorems can be proved in much greater generality by using 
the categorical approach of
Knus and Quebbeman-Scharlau-Schulte~\cite{knus_quadratic_1991,quebbemann_quadratic_1979}. A disadvantage of that approach is that the hands on nature of Classification Theorem~\ref{thm:MainLinkingForm} becomes less apparent in this categorical framework.

\begin{proof}[Proof of Classification Theorem~\ref{thm:MainLinkingForm}]
We present a proof of the Classification Theorem with a very down-to-earth approach. The proof will follow and generalize the approach
of \cite[Section 4]{BorodzikFriedl2}. Sometimes we will refer to that paper for statements that resemble ours.
  The proof comprises the following steps. We describe these steps now, with references to precise statements in the article.
  \begin{enumerate}[label=(I-\arabic*)]
    \item In Lemma~\ref{lem:orthogonal} we show that~$(M,\pairing)$ splits into~an orthogonal sum of~$\basicF_\xi$-primary parts~$(M_\xi,\pairing_\xi)$ as~$\xi$ ranges through~$\Xi^\F$; here recall that an $\LF$-module $H$ is called $\basicF_\xi$-primary if for every $x \in H$, there exists an integer $n \geq 0$ such that $\basicF_\xi^nx=0$.
    \label{item:orthogonal}
  \item In Subsection~\ref{sub:small_rank} we provide a classification of forms on cyclic modules and on~$\module(n,\xi,\R)^{\oplus 2}$ for~$\xi\in\{-1,1\}$ and~$n$ odd.
    \label{item:small}
  \item In Subsection~\ref{sub:splitting} we show that~$(M_{\xi},\pairing_{\xi})$ can be presented as an orthogonal sum of forms supported on cyclic modules: this is possible unless~$\xi\in\{-1,1\}$,
   ~$\F=\R$. In the latter case,~$(M_{\xi},\pairing_{\xi})$ is an orthogonal sum of forms on cyclic modules and copies of~$\ff(n,\xi,\R)$,
   see Subsection~\ref{sub:splitting}.
    \label{item:further_orthogonal}
  \item In Subsection~\ref{sub:linear} we prove that the splitting obtained in item~\ref{item:further_orthogonal} is unique up to permuting summands.\label{item:uniqueness}
\end{enumerate}
Given these four results we conclude the proof of Classification Theorem~\ref{thm:MainLinkingForm} in the following way.
A pairing $(M,\pairing)$ is decomposed first into $\basicF_\xi$-primary summands via \ref{item:orthogonal}, and next into
forms supported on $\module(n,\xi,\F)$ (except for the part $\ff(\xi,n,\R)$ for $n$ even) via \ref{item:further_orthogonal}.
By item~\ref{item:small}, each form on $\module(n,\xi,\F)$ is either $\ee(\xi,n,\eps,\F)$ or $\ff(\xi,n,\F)$, depending on whether
$\xi\in\Xi^\F_\ee$ or $\xi\in\Xi^\F_\ff$. For $\xi\in\{-1,1\}$ and $\F=\R$, non-degenerate forms are supported only on $\module(n,\xi,\F)$
for $n$ odd, and these are $\ee(n,\xi,\eps,\R)$.
In this way we show, that each form $(M,\pairing)$ can be presented as an orthogonal sum of forms of type $\ee$ and~$\ff$.

To prove uniqueness, we note that the decomposition of $(M,\pairing)$ into $\basicF_\xi$-primary parts is unique. The uniqueness of the
decomposition of $(M_\xi,\pairing_\xi)$ is guaranteed by item~\ref{item:uniqueness}.
\end{proof}
The following result proves the~\ref{item:orthogonal} part of the proof of Classification Theorem~\ref{thm:MainLinkingForm}.
\begin{lemma}\label{lem:orthogonal}
  Let~$(M,\pairing)$ be a linking form. There exists an orthogonal decomposition~$(M,\pairing)=\bigoplus_{\xi\in\Xi^\F} (M_\xi,\pairing_\xi)$
  where~$M_\xi$ is~$\basicF_\xi$-primary.
\end{lemma}
\begin{proof}
  Decompose~$M$ as a direct sum of~$\basicF_\xi$-primary modules~$M_\xi$.
  It is straightforward to prove that the decomposition~$M=\bigoplus M_\xi$ is orthogonal with respect to~$\pairing$,
  compare \cite[Lemma 3.1]{milnor_isometries}, \cite[first claim on page 672]{BorodzikFriedl2} and
  also \cite[Section II.5]{knus_quadratic_1991} for a more general perspective.
  We let~$\pairing_\xi$ be the restriction of~$\pairing$
  to~$M_\xi$.   
\end{proof}
\begin{remark}
  It is clear that the decomposition $M=\bigoplus M_\xi$ is unique, that is, the $\basicF_\xi$-primary part of $M$ is well-defined.
  Then, $\pairing_\xi$ is defined intrinsically, and so the pairing $(M_\xi,\pairing_\xi)$ is defined uniquely.
\end{remark}

\subsection{Classification of forms on modules of small ranks}\label{sub:small_rank}
In this subsection we deal with item~\ref{item:small} of the proof of Classification Theorem~\ref{thm:MainLinkingForm}. The classification result for modules of small ranks is summarized by the following statement.
\begin{theorem}\label{thm:elementary_classification}\
  \begin{enumerate}[label=(C-\arabic*)]
    \item For~$\xi\in \Xi^\R_{\ee+}$, any non-degenerate form on~$\module(n,\xi,\R)$ is isomorphic either to~$\ee(n,+1,\xi,\R)$ or~$\ee(n,-1,\xi,\R)$;\label{item:form_e_R}
  \item For~$\xi=\pm 1$, the module~$\module(n,\xi,\R)$ supports a non-degenerate form if and only if~$n$ is even. In that case,
    there are two non-degenerate forms on~$\module(n,\xi,\R)$, namely~$\ee(n,+1,\xi,\R)$ and~$\ee(n,-1,\xi,\R)$;\label{item:form_e_1}
  \item For~$\xi\in \Xi^\C_\ee$, any non-degenerate form on~$\module(n,\xi,\C)$ is isomorphic either to~$\ee(n,+1,\xi,\C)$ or~$\ee(n,-1,\xi,\C)$;\label{item:form_e_C}
  \item For~$\xi\in \Xi^\F_\ff\setminus\{-1,1\}$, 
    any non-degenerate form on~$\module(n,\xi,\F)$ is isomorphic to~$\ff(n,\xi,\F)$;\label{item:form_f_F}
  \item For~$\xi=\pm 1$ and~$n$ odd, the module~$\module(n,\xi,\R)\oplus \module(n,\xi,\R)$ supports a unique non-degenerate form,  namely~$\ff(n,\xi,\R)$.\label{item:form_f_1}
  \end{enumerate}
\end{theorem}
\begin{proof}
Despite this result having not appeared in the literature before, it can be parsed together from previously known results.

\textbf{Items~\ref{item:form_e_R}, \ref{item:form_e_1} and~\ref{item:form_e_C}.}
We begin with the cases involving~$\ee$-forms, which are all proved essentially by the same method. 
  The techniques of Milnor \cite[Section 3]{milnor_isometries} for
  isometric structures over~$\F$ apply similarly to
  linking forms, see \cite{LevineAlgebraic} and \cite[Theorem 1.3]{LevineMetabolicHyperbolic}.
  The bottom line is that both isometric structures over~$\F$ and linking forms over~$\LF$ are classified by a collection of associated Hermitian forms over~$\F$; see also Subsection~\ref{sub:linear}. More precisely,
one shows that each of the  modules~$\module(n,\xi,\F)$ in the above list carries precisely two non-degenerate forms.  
While this statement can be deduced from~\cite[Reduction Theorem II.4.6.1]{knus_quadratic_1991} or \cite[Theorem II.2]{quebbemann_quadratic_1979}, we instead outline an approach that requires less background.
  Once we know that~$\module(n,\xi,\F)$ supports two non-degenerate forms, the classification is then completed by noting~$\ee(n,+1,\xi,\F)$ is not isomorphic to~$\ee(n,-1,\xi,\F)$; this is proved in Lemma~\ref{lem:not_iso_ee_C} for the~$\F=\C$ case, the other cases are analogous.

Our approach is based on the one from
~\cite[Claim on page 669]{BorodzikFriedl2} that concerns the case where~$\F=\R$ and~$\xi\neq 1$.
  One verifies that any pairing on~$\module(n,\xi,\R)$ is given by  
 ~$(x,y)\mapsto\frac{qxy^\#}{\basicR_\xi^n}$ for some polynomial~$q$. 
Next, using methods of \cite[Section 3.2]{BorodzikFriedl2}, one proves that~$q=\epsilon gg^\#\bmod\basicR_\xi^n$, for some polynomial~$g$ and
sign choice~$\epsilon$. This leads to the forms~$\ee(n,\xi,\pm 1,\R)$. This proves~\ref{item:form_e_R}.

  While not explicitly stated in \cite{BorodzikFriedl2},
  this proof works for~$\xi=\pm 1$,~$\R=\xi$ as long as~$n$ is even, provided that one presents
  the pairing as
  ~$(x,y)\mapsto\frac{qxy^\#}{(t-\xi)^{n/2}(t^{-1}-\xi)^{n/2}}$. This takes care of~\ref{item:form_e_1}.

  The same argument works also for forms over~$\C$ for~$\xi\in\Xi^\C_\ee$, that is, item~\ref{item:form_e_C}.
  A minor modification is needed when~$n$ is odd and~$\F=\C$. 
  In this case,  an arbitrary linking form is given by
  \begin{equation}\label{eq:form_with_q}
  (x,y)\mapsto\frac{qxy^\#}{(t-\xi)^{(n+1)/2}(t^{-1}-\ol{\xi})^{(n-1)/2}},
\end{equation} 
where, by symmetry,~$q$ is either~$\xi$-positive or~$\xi$-negative.  Suppose $\pairing_1$ and $\pairing_2$
are two forms as in \eqref{eq:form_with_q} with the role of $q$ played by two polynomials $q_1$ and $q_2$.
The methods of \cite[Section 3.2]{BorodzikFriedl2} show
  that if~$q_1,q_2$ are both~$\xi$-positive or both~$\xi$-negative, then~$q_1/q_2$ is a norm, that is~$q_1/q_2=gg^\#$ modulo
 ~$\basicF_\xi^n$. That is, if $q_1$ and~$q_2$ are both $\xi$-positive, the forms $\pairing_1$ and $\pairing_2$ are isometric. Likewise,
 if $q_1$ and $q_2$ are both $\xi$-negative, the forms $\pairing_1$ and $\pairing_2$ are isometric. If one of the $q_1,q_2$ is
 $\xi$-positive and the other one is $\xi$-negative, the forms $\pairing_1$ and $\pairing_2$ are not isometric (the proof of Lemma~\ref{lem:not_iso_ee_C} can be used to show this).

 This argument shows that forms corresponding to~$\xi$-positive polynomials are all isometric and forms corresponding
 to~$\xi$-negative polynomials constitute the second isometry class. 
  We leave the remaining details to the reader.

  \smallskip
  \textbf{Item~\ref{item:form_f_F}} As a next step we deal with~$\ff$-type forms. These can be approached from various perspectives. The most enlightening is probably the one
  using a general approach of \cite[Section II.6.4]{knus_quadratic_1991}; we will sketch it before giving a more direct argument.
  The module~$M=\module(n,\xi,\F)$ can be presented as a sum of
  two modules~$M=M_+\oplus M_-$, where~$M_+=\LC/(t-\xi)^n$,~$M_-=\LC/(t^{-1}-\xi)^n$ if~$\F=\C$ and a similar formula holds for~$\F=\R$ (see also
  \cite[second claim on page 672]{BorodzikFriedl2}, where a detailed description of the case~$\F=\R$ is given).
  A direct argument shows that~$M_{\mp}\cong\Hom_{\LF}(M_{\pm};\OF/\LF)$. It already follows from \cite[Proposition II.6.4.1]{knus_quadratic_1991} that the pairing
  is determined by the underlying module structure.  This implies \ref{item:form_f_F}.

  A more direct approach to~\ref{item:form_f_F} relies on the methods of \cite{BorodzikFriedl2}.
  We write the pairing as~$(x,y)\mapsto\frac{qxy^\#}{\basicF_\xi^n}$ for some~$q\in\LF$. Then, as~$\basicF_\xi$ has fixed sign
  on~$S^1$, there exists a polynomial~$g\in\LF$ such that~$q=gg^\#\bmod\basicF_\xi^n$: this is proved for~$\F=\R$ in \cite[Lemma 3.2 and Proposition 3.3]{BorodzikFriedl2},
  but the proof carries over to~$\C$ without any significant changes. The equality~$q=gg^\#\bmod\basicF_\xi^n$ means that all forms are isometric to the form
  with~$q\equiv 1$.

  \smallskip
  \textbf{Item~\ref{item:form_f_1}.}
  As the statements are not completely standard, we present
  more details of proofs.
  We begin with the following result.
  \begin{lemma}\label{lem:does_not_support}
    Suppose~$\F=\R$ and~$\xi\in\{-1,1\}$.
    The module~$\module(n,\xi,\R)$ does not support
    a non-degenerate form if~$n$ is odd.
  \end{lemma}
  \begin{proof}[Proof of Lemma~\ref{lem:does_not_support}]
    Suppose towards contradiction that~$(\module(n,\xi,\R),\pairing)$ is a non-degenerate form with~$\xi=\pm 1$ and~$n$ odd.
    Let~$r$ be a polynomial such that~$\pairing(1,1)=\frac{r}{(t-\xi)^n}$. Consider the polynomial~$\wt{r}=r t^{-(n-1)/2}$ and observe that
    \[\frac{r}{(t-\xi)^{n}}=\frac{\wt{r}}{(t-\xi)((t-\xi)(t^{-1}-\ol{\xi}))^{(n-1)/2}}.\]
    Since both~$\pairing(1,1)$ and~$((t-\xi)(t^{-1}-\ol{\xi}))^{(n-1)/2}$ are symmetric, it follows that~$\wt{r}/(t-\xi)$ must also be symmetric.
    From this, we deduce that
   ~$-\xi t \wt{r}(t^{-1})=\wt{r}(t).$ 
    Putting \(t = \xi\) in the latter equality yields~$\wt{r}(\xi) = -\wt{r}(\xi)$.
    Thus, \(\xi\) is a root of \(r\), contradicting the assumption that~$r$ is coprime to~\(\basicR_\xi(t)\).
  \end{proof}
  The second result classifies forms on~$\module(n,\xi,\R)^{\oplus 2}$. The statement is stronger than needed for the proof of item~\ref{item:form_f_1}, because we will need it in the future.
  \begin{lemma}\label{lem:indeed_supports}
    Suppose~$\xi=\pm 1$ and~$\F=\R$. 
    Consider a form~$(M,\pairing)$ with 
   ~$M=\module(n,\xi,\R)\oplus \module(n,\xi,\R)$. Assume~$x=(1,0)\in M$ generates one of the summands and~$\basicR_\xi^{n-1}\pairing(x,x)=0$.
    Then, there exists a basis~$\wt{x},\wt{y}$ of~$M$ such that~$\pairing(\wt{x},\wt{x})=\pairing(\wt{y},\wt{y})=0$ and 
   ~$\pairing(\wt{x},\wt{y})=1/\basicR_\xi^n$, that is to say,~$M$ is isometric to~$\ff(n,\xi,\R)$.
  \end{lemma}
  Lemma~\ref{lem:indeed_supports} should be standard, but we were unable to locate the precise statement in the literature.
  We delay its proof until the end of the proof of Theorem~\ref{thm:elementary_classification}, which we
  now complete. 

  Suppose~$\F=\R$,~$n$ is odd and~$\xi\in\{-1,1\}$. Consider the module~$M=\module(n,\xi,\R)^{\oplus 2}$ and take~$x=(1,0)$. 
  If~$\basicF_\xi^{n-1}\pairing(x,x)\neq 0$, then
  the submodule generated by~$x$ supports a non-degenerate pairing, contradicting Lemma~\ref{lem:does_not_support}.
  Therefore,~$\basicF_\xi^{n-1}\pairing(x,x)=0$.
  We now apply Lemma~\ref{lem:indeed_supports} to conclude that the pairing~$\pairing$ is isometric to~$\ff(n,\xi,\R)$.
 This concludes the proof of Lemma~\ref{lem:does_not_support}.
\end{proof}
\begin{proof}[Proof of Lemma~\ref{lem:indeed_supports}]
  Take~$x=(1,0)\in M$. 
  As~$\pairing$ is non-singular, there exists~$y\in M$ such that~$\pairing(x,y)=\frac{1}{\basicR_\xi^n}$. Clearly,~$y$ is not a multiple of~$x$, for otherwise~$\basicR_\xi^{n-1}\pairing(x,x)\neq 0$ contradicting the assumption. Note that $y$ is not equal to $\basicR_\xi z$ for
  any $z$, because then $\pairing(x,z)=\frac{1}{\basicR_\xi^{n+1}}$, which is impossible.
  Since $\LR$ is a PID,
this means that $y$ generates a free summand of $M$.
This summand 
intersects trivially the summand generated by $x$. In fact, if $px=qy$ for some polynomials $p$ and $q$, then we write
$p=p_0\basicF_\xi^r$, $q=q_0\basicF_\xi^s$ for $p_0,q_0$ coprime with $\basicF_\xi$. 
As neither $x$ and $y$ can be written as $\basicF_\xi z$ for $z\in M$, we have $r=s$, that is, $p_0x=q_0y$. 
This contradicts
conditions $\basicF_\xi^{n-1}\pairing(x,y)\neq 0=\basicF_\xi^{n-1}\pairing(x,x)$.
As the cyclic submodules generated by $x$ and $y$ 
intersect trivially 
we conclude that $x,y$ generate $M$. 

  Let~$q_x$ and~$q_y$ be such that~$\pairing(x,x)=\frac{q_x}{\basicR_\xi^n}$,~$\pairing(y,y)=\frac{q_y}{\basicR_\xi^n}$.
  By symmetry,
  \begin{equation}\label{eq:sym_q}q_x^\#=-\xi^n t^nq_x,\  q_y^\#=\xi^n t^n q_y.\end{equation}
  We strive to improve the basis in such a way that~$q_x=q_y=0$. 
  To this end, denote by~$k_x$ and~$k_y$ the integers such that~$\basicR_\xi^{k_x}$ divides~$q_x$,
  but~$\basicR_\xi^{k_x+1}$ does not, and the same for~$k_y$ and~$q_y$.  Suppose~$k_x$ or~$k_y$ are less than~$n$. 
  Assume that~$k_x\le k_y$: otherwise we switch the roles
  of~$x$ and~$y$. Write~$\wt{q}_x=q_x/\basicR_\xi^{k_x}$ and~$\wt{q}_y=\basicR_\xi^{k_x}$.

  We aim to make a base change~$x'=x+sy$, where~$s=\basicR_\xi^{k_x}\wt{s}$ and~$s\in\LR$ will be defined in such a way, that~$k_x$
  increases. To this end, we 
  let~$q_{x'}$ be such that~$\pairing(x',x')=\frac{q_{x'}}{\basicR_\xi^n}$. We have
  \[\pairing(x',x')=\pairing(x+sy,x+sy)=\frac{q_x+ss^\# q_y+s-\xi^n t^n s^\#}{\basicR_\xi^n}.\]
  The expression in the numerator can be written as
  \[q_x+ss^\# q_y-\xi^n t^n s^\#=\basicR_\xi^{k_x}(\wt{q}_x+(\wt{s}+t^{n-k_x}\wt{s}^\#)+\basicR_\xi^{k_x}\wt{s}s^\#\wt{q}_y).\]
  Let~$P=\wt{q}_x+(\wt{s}+t^{n-{k_x}}\wt{s}^\#)+\basicR_\xi^{k_x}\wt{s}s^\#\wt{q}_y$. Note that as~$\xi=\pm 1$, for any~$f\in\LR$
  we have~$f(\xi)=f^\#(\xi)$. As~$n-k_x$ is even, we have:
  \[P(\xi)=\wt{q}_x(\xi)+2\wt{s}(\xi).\]
  We choose now~$\wt{s}$ in such a way that~$\wt{s}(\xi)=-\frac12\wt{q}(\xi)$. Then~$\basicR_\xi^{k_x+1}$ divides~$q_{x'}$.
  Note that 
  \[\pairing(x',y)=\frac{1+sq_y}{\basicR_\xi^n}\]
  As~$1+sq_y$ is invertible modulo~$\basicR_\xi^n$, there exists~$y'$, a multiple of~$y$, such that~$\pairing(x',y')=\frac{1}{\basicR_\xi^n}$.
  Replacing~$(x,y)$ by a new basis~$(x',y')$ has the effect of increasing~$k_x$ while preserving~$k_y$. Acting inductively, we arrive
  at the situation, where~$k_x,k_y\ge n$, that is~$\pairing(x,x)=\pairing(y,y)=0\in\R(t)/\LR$. 
\end{proof}

\subsection{The splitting theorem}\label{sub:splitting}
The following result, Theorem~\ref{thm:first_splitting}, has three parts. The first two parts prove item~\ref{item:further_orthogonal}
of Classification Theorem~\ref{thm:MainLinkingForm}. The last part of Theorem~\ref{thm:first_splitting} proves item~\ref{item:uniqueness}.
While it is convenient to state these results in one theorem, the proof of the last item of Theorem~\ref{thm:first_splitting} is
deferred to Subsection~\ref{sub:linear}.

\begin{theorem}\label{thm:first_splitting}
  Let~$(M,\pairing)$ be a non-degenerate linking form over~$\LF$,
  and~$M$ is~$\basicF_\xi$-primary for some~$\xi$. 
  Then 
  \begin{itemize}
    \item[(1)]~$(M,\pairing)$ splits as an orthogonal sum of forms~$(M_{n,\xi},\pairing_{n,\xi})$, 
  with~$n\ge 1$, where $\pairing_{n,\xi}$ is supported on a direct sum of copies of~$\module(n,\xi,\F)$.
\item[(2)] Each of the~$(M_{n,\xi},\pairing_{n,\xi})$ is a direct sum of copies of~$\ee(n,\xi,\pm 1,\F)$ or
 ~$\ff(n,\xi,\F)$ (where~$n,\xi$ determine uniquely, whether it is a sum of~$\ee$-type forms or~$\ff$-type forms).
\item[(3)] The decompositions in (1) and (2) are unique, in (2) up to permuting summands.
  \end{itemize}
\end{theorem}
The splitting result can be deduced from an analogous result for
isometric structures, see \cite[Theorem 3.2]{milnor_isometries}.
  For the reader's convenience, we sketch a quick argument, which proves both (1) and (2). Item~(3) will be proved in Subsection~\ref{sub:linear}.
  \begin{proof}[Proof of items~(1) and (2) of Theorem~\ref{thm:first_splitting}]
  For any~$x\in M$ we denote by~$k_x$ the minimal exponent such that~$\basicF_\xi^{k_x}x=0\in M$. As~$M$ is finitely generated, all these~$k_x$ are bounded, and there exists~$x$ that maximizes~$k_x$. As~$\pairing$ is non-degenerate, there exists~$y$ such that
 ~$\pairing(x,y)=\frac{q}{\basicF_\xi^{k_x}}$ for~$q$ coprime with~$\basicF_\xi$. Indeed, if not,~$\basicF_\xi^{k_x-1}x$ pairs trivially
  with all elements in~$M$, contradicting the non-degeneracy of~$\pairing$. Note that~$q$ is invertible modulo~$\basicF_\xi^{k_x}$,
  hence, on multiplying~$y$ by~$q^{-1}$, we may and will assume that
\begin{equation}
\label{eq:generators}
\pairing(x,y)=\frac{1}{\basicF_\xi^k}.
\end{equation}
Informally, the ideal case would be that $\pairing(x,x)=\frac{q}{\basicF_\xi^k}$ for some $q$ coprime with $\basicF_\xi$.
  Then, we could repeat the argument of \cite[page 671]{BorodzikFriedl2}
   to show that the submodule of $M$ generated by $x$ splits off as an orthogonal summand. The next result
  shows that, with one exception, we can find an element that generates a summand that splits off.
  \begin{lemma}\label{lem:split_if_possible} 
    Unless~$\xi=\pm 1$,~$\F=\R$ and~$k=k_x$ odd, at least one of the elements~$z\in\{x,y,x+y,x+yi\}$ 
    satisfies~$\pairing(z,z)=\frac{q}{\basicF_\xi^k}$
    for~$q$ coprime with~$\basicF_\xi$ (the case~$z=x+yi$ is considered only if~$\F=\C$).
  \end{lemma}
  We postpone the proof of Lemma~\ref{lem:split_if_possible} until the end of the proof of Theorem~\ref{thm:first_splitting}.
  The missing case of Lemma~\ref{lem:split_if_possible} is dealt with in the next result.
  \begin{lemma}\label{lem:split_if_not_possible}
    In case~$\xi=\pm 1$,~$\F=\R$ and~$k$ odd, the module generated by~$x$ and~$y$ is isometric to~$\ff(k,\xi,\R)$.
  \end{lemma}
  \begin{proof}[Proof of Lemma~\ref{lem:split_if_not_possible}]
    We must have~$\basicF_\xi^{k-1}\pairing(x,x)=0$, for otherwise the form~$\pairing$ on the cyclic summand generated by~$x$ is non-degenerate, contradicting Lemma~\ref{lem:does_not_support}.
In particular~$x,y$ generate a rank two submodule of~$M$.; call it~$M_{xy}$. 
By Lemma~\ref{lem:indeed_supports}, the pairing~$\pairing$ on~$M_{xy}$ is isometric to~$\ff(k,\xi,\R)$.
This concludes the proof of Lemma~\ref{lem:split_if_not_possible}
  \end{proof}
  We continue the proof of Theorem~\ref{thm:first_splitting}. Let~$M_0$ be the submodule generated by~$x$ (in the case of Lemma~\ref{lem:split_if_possible}) and by~$x,y$ if~$\xi=\pm 1$,~$\F=\R$ and~$k=k_x$ odd, that is, in the case of Lemma~\ref{lem:split_if_not_possible}.
  An argument based on Gauss orthogonalization, compare \cite[page 671]{BorodzikFriedl2}, shows that there exists a module~$M'$ such that~$M=M_0\oplus M'$
  and the decomposition is orthogonal with respect to~$\pairing$. 

  The linking form~$\pairing|_{M_0}$ on~$M_0$ is non-degenerate and~$(M_0,\pairing|_{M_0})$ is one from the list in Theorem~\ref{thm:elementary_classification}. Hence, it is one of the~$\ee$ or~$\ff$ forms. On the other hand, the rank of~$M'$ is strictly less than the rank of~$M$. We proceed by induction to show that~$(M,\pairing)$ is a direct sum of~$\ee$ and~$\ff$ forms.
  This concludes the proof of  the existence part of Theorem~\ref{thm:first_splitting} modulo the proof of Lemma~\ref{lem:split_if_possible}.
\end{proof}
\begin{proof}[Proof of Lemma~\ref{lem:split_if_possible}]
While the proof is 
similar to~\cite[Claim in the proof of Lemma 4.4]{BorodzikFriedl2}, it gives insight into the importance of the assumptions. Suppose~$\basicF_\xi^{k-1}\pairing(x,x)=\basicF_\xi^{k-1}\pairing(y,y)=0\in\F(t)/\LF$, otherwise there is nothing to prove. 
Using~\eqref{eq:generators},  we have
    \begin{multline*}\basicF_\xi^{k-1}\pairing(x+y,x+y)=\basicF_\xi^{k-1}(\pairing(x,x)+\pairing(y,y))+\basicF_\xi^{k-1}(\pairing(x,y)+\pairing(y,x))=\basicF_\xi^{k-1}(\pairing(x,y)+\pairing(y,x))=\\
    =\basicF_\xi^{k-1}(\frac{1}{\basicF_\xi^k}+\frac{1}{(\basicF_\xi^k)^\#})\in\F(t)/\LF.\end{multline*}
    If~$\basicF_\xi=\basicF_\xi^\#$ we are done, since then~$\basicF_\xi^{k-1}(\frac{1}{\basicF_\xi^k}+\frac{1}{(\basicF_\xi^k)^\#})=2/\basicF_\xi\neq 0\in\OF/\LF$. 
    The case~$\basicF_\xi\neq\basicF_\xi^\#$ splits into two cases.
    \begin{itemize}
    \item[(a)]~$\F=\R$ and~$\xi=\pm 1$. Then~$\basicF_\xi=(t\pm 1)$,~$\basicF_\xi^\#=-\xi t^{-1}\basicF_\xi$, so
      \[\frac{1}{\basicF_\xi^k}+\frac{1}{(\basicF_\xi^k)^{\#}}=\frac{1+(-\xi)^{-k} t^{k}}{\basicF_\xi^k}.\]
      If~$k$ is even,~$(-\xi)^{-k}t^{k}+1$ is coprime with~$\basicF_\xi$, so indeed~$\basicF_\xi^{k-1}\pairing(x+y,x+y)\neq 0$. 
    \item[(b)]~$\F=\C$ and~$\xi\in\Xi^\C_{\ee}$. Then~$\basicF_\xi^\#=-\ol{\xi} t^{-1}\basicF_\xi$. 
      Then 
      \[\frac{1}{\basicF_\xi^k}+\frac{1}{(\basicF_\xi^{k})^\#}=\frac{1+(-\xi)^kt^k}{\basicF_\xi^k}.\]
      If~$(-\xi)^kt^k+1$ does not evaluate to~$0$ at~$t=\xi$, then we are done. Otherwise, repeating the calculations we show that~$\basicF_\xi^{k-1}\pairing(x+yi,x+yi)=(i-i(-\xi)^kt^k)/\basicF_\xi$. This expression is not zero at~$t=\xi$ if~$(-\xi)^kt^k+1$ vanishes at~$t=\xi$.
    \end{itemize}
    The two items conclude the proof of Lemma~\ref{lem:split_if_possible}.
  \end{proof}

  \subsection{Uniqueness of the splitting}\label{sub:linear}
  We now move on to uniqueness of splitting, that is, item~(3) of Theorem~\ref{thm:first_splitting} and item~\ref{item:uniqueness}
  of Classification Theorem~\ref{thm:MainLinkingForm}.
  The uniqueness can be deduced from \cite[Theorem 1.3]{LevineMetabolicHyperbolic}, however
  in that paper only the case of~$\F=\R$ is given. 
  We will sketch the argument working for~$\F=\C$ as well. 
  Other proofs can be extracted from~\cite[Theorem 3.3]{milnor_isometries} and \cite{LevineAlgebraic}.

  The methods used in this section, that is, passing from
  a linking form to a Hermitian form on a vector space over~$\F$, can be viewed
  as an instance of Transfer Theorem~\cite[Section II.3, especially Theorem II.3.3.1]{knus_quadratic_1991}. 

  Suppose~$(M,\pairing)$ is a linking form on a module~$M=\module(n,\xi,\F)^{\oplus k}$ for some~$k$. 
  We make the following observation.
  \begin{lemma}\label{lem:underlying}
  If~$\xi\in\Xi_\ff^\F\setminus\Xi_\ee^\F$,
  the isometry class of $(M,\lambda)$ is determined by the underlying module structure.
\end{lemma}
\begin{proof}
  The result is well-known, see e.g. \cite[Proposition II.6.4.1]{knus_quadratic_1991}. We sketch a quick argument for the reader's convenience.
  Let~$(M,\pairing)$ be a linking form over a~$\basicF_\xi$-primary module. 
  Decompose~$M$ as a direct sum~$M=M_1\oplus\dots\oplus M_s$, where~$M_i=\module(i,\xi,\F)^{\oplus k_i}$, for \(i=1,2,\ldots,s\).
  The numbers~$k_1,\dots,k_s$ are invariants of the module structure.

  On the other hand, Theorem~\ref{thm:first_splitting} shows that~$(M,\pairing)$ decoposes as an orthogonal sum of linking forms~$\ff(i,\xi,\F)$. Let $\wt{k}_i$ be the number of times the linking forms~$\ff(i,\xi,\F)$ appears in this decomposition.
The underlying module for $\ff(i,\xi,\F)$
  is $\module(i,\xi,\F)$,  so the underlying module of $\bigoplus \ff(i,\xi,\F)^{\oplus\wt{k}_i}$ is
  $\bigoplus\module(i,\xi,\F)^{\oplus\wt{k}_i}$.

  The decomposition of $M$ into summands $M_1\oplus\dots\oplus M_s$ is unique up to isomorphism of summands. That is to say, $k_i=\wt{k}_i$.
  Therefore, the module structure determines the linking form up to isometry.
\end{proof}

From now on, we consider the case~$\xi\in\Xi_\ee^\F$. 
The space~$H_M:=M/\basicF_\xi M$ has the structure of a vector space over~$\F$, its dimension is equal to~$k\deg\basicF_\xi$.
  In \cite[Theorem 3.3]{milnor_isometries} and \cite[Section 1]{LevineMetabolicHyperbolic},
  it is shown that the form~$\pairing$ determines an $\F$-valued Hermitian form on~$H_M$, which we
  denote by~$\pairing_{H_M}$. 
  We give a quick overview of the construction.
  \begin{example}
  Suppose~$(M,\pairing)$ is a linking form on~$M=\module(n,\xi,\F)^{\oplus k}$ for some~$k$. 
We describe how to associate a Hermitian form $(H_M,\lambda_{H_M})$ over $\F$ to $(M,\lambda)$, where $H_M:=M/\basicF_\xi M$.
\begin{itemize}
\item  Suppose~$n$ is odd,~$\F=\C$ and~$\xi\in\Xi^\C_\ee$.
  The Hermitian form
  \begin{equation}\label{eq:HM}
    \begin{split}
      \wt{\pairing}_H&\colon M\times M\to\C\\
      (x,y)&\mapsto i\ol{\xi}\left.(t-\xi)^{(n-1)/2-1}(t^{-1}-\ol{\xi})^{(n-1)/2}\pairing(x,y)\right\vert_{t=\xi},
    \end{split}
  \end{equation}
  where the subscript denotes the evaluation,  descends to a non-degenerate Hermitian form~$\pairing_{H_M}\colon H_M\times H_M\to\C$.
\item  If~$\xi\in\Xi^\F_\ee$ and either~$n$ is even or~$\F=\R$, we define the pairing on~$H_M$ starting with the Hermitian form
  \begin{equation}\label{eq:HM2}
    \begin{split}
      \wt{\pairing}_H&\colon M\times M\to\C\\
      (x,y)&\mapsto \left.\basicF_\xi^n\pairing(x,y)\right\vert_{t=\xi},
      \end{split}
  \end{equation}
  which induces a form~$\pairing_{H_M}$ 
  on the quotient~$H_M=M/\basicF_\xi M$.
\end{itemize}
  In both cases, the form~$\pairing_{H_M}$ has signature~$\epsilon$ if~$(M,\pairing)=\ee(1,\xi,\epsilon,\C)$; see Lemma~\ref{lem:residue} for the case of \eqref{eq:HM}.
  If $\xi=\pm 1$ and $\F=\R$, the form $\pairing_{H_M}$ is defined in the same way, that is, via \eqref{eq:HM} for $n$ odd and \eqref{eq:HM2}
  for $n$ even. The forms $\pairing_{H_M}$ associated with $\xi=\pm 1$ with $\F=\R$ are hyperbolic as real forms on a vector space.
\end{example}

\begin{lemma}\label{lem:uniqueness_for_ee}
  Suppose~$\xi\in\Xi^\F_\ee$. The linking forms
  ~$(M,\pairing)=\ee(n,\xi,+1,\F)^{\oplus a}\oplus\ee(n,\xi,-1,\F)^{\oplus b}$ and 
  ~$(M',\pairing')=\ee(n,\xi,+1,\F)^{\oplus a'}\oplus\ee(n,\xi,-1,\F)^{\oplus b'}$ are isometric if and only if~$a=a'$ and~$b=b'$.
\end{lemma}
\begin{proof}
  The `if' part is clear. We prove the `only if' part. Suppose that $\phi\colon(M,\pairing)\cong(M',\pairing')$ is an isometry.
  Consider the Hermitian forms $(H_M,\pairing_{H_M})$ and $(H_{M'},\pairing_{H_{M'}})$. 

  \emph{Claim.} The isometry $\phi$ induces an isometry between $(H_M,\pairing_{H_M})$ and $(H_{M'},\pairing_{H_{M'}})$.

  Given the claim, we note that $(H_M,\pairing_{H_M})$ is diagonal with $a$ terms equal to $+1$ and $b$ terms equal to $-1$.
  Likewise, $(H_{M'},\pairing_{H_{M'}})$ is diagonal with $a'$ terms equal to $+1$ and $b'$ terms equal to $-1$.
  Existence of an isometry means that $a=a'$ and $b=b'$.

  \smallskip
  It remains to prove the claim. It follows directly from the description of the pairing in \eqref{eq:HM} and \eqref{eq:HM2}. To be more
  precise, an isometry of $(M,\lambda)$ and $(M',\lambda')$ clearly induces an isometry between the associated $\F$-valued forms $\wt{\lambda}_M$
  and $\wt{\lambda}_{M'}$. As $\basicF_\xi M$ is mapped to $\basicF_\xi M'$, the isometry between $\wt{\lambda}_M$ and $\wt{\lambda}_{M'}$
  induces an isometry between $\lambda_{H_M}$ and $\lambda_{H_{M'}}$.
\end{proof}

We use the Hermitian forms~$H_M$ to prove the following result.
\begin{proposition}\label{prop:uniqueness_two}
  Assume~$\xi\in\Xi^\F_\ee$.
  Suppose~$(M,\pairing)$ and~$(M',\pairing')$ are forms over~$\basicF_\xi$-primary modules that decompose as orthogonal sums:
\begin{align*}
  (M,\pairing)&=(M_1,\pairing_1)\oplus\dots\oplus (M_s,\pairing_s)\\
  (M',\pairing')&=(M'_1,\pairing'_1)\oplus\dots\oplus (M'_{s'},\pairing_{s'}),
\end{align*}
where~$(M_i,\pairing_i)$ and~$(M'_i,\pairing'_i)$ are sums of copies of~$\module(i,\xi,\F)$. Let~$H_i=H_{M_i}$ and~$H'_i=H_{M'_i}$ be
the Hermitian forms respectively associated with~$(M_i,\pairing_i)$ and~$(M'_i,\pairing'_i)$.

An isometry~$\phi\colon(M,\pairing)\to(M',\pairing')$ induces an isometry between~$(H_i,\pairing_{H_i})$ and~$(H'_i,\pairing_{H'_i})$.
\end{proposition}
\begin{proof}
  Consider~$x,y\in M_i$ and use~$\wt{x},\wt{y}$ to denote their classes in~$M_i/\basicF_\xi M_i$. 
  Let~$x'=\phi(x),y'=\phi(y)$. 
  We let~$x'_1,\dots,x'_{s'}$ and~$y'_1,\dots,y'_{s'}$.
  Assume that~$\pairing(x,y)$ is given via \eqref{eq:HM2} as opposed to \eqref{eq:HM}, the other case leads to different looking formulae,  but the proof is the same. 
Since~$\phi$ is an isometry,  we have
  \begin{equation}\label{eq:pairing_split}\pairing_{H_{M_i}}(\wt{x},\wt{y})=\left.\basicF_\xi^{i}\pairing(x,y)\right\vert_{t=\xi}=\left.\basicF_\xi^{i}\pairing'(x',y')\right\vert_{t=\xi}=\left.
  \sum_{j=1}^{s'}\basicF_\xi^i\pairing_j'(x'_j,y'_j)\right|_{t=\xi}.\end{equation}
  \begin{lemma}\label{lem:iszero}\
    \begin{itemize}
      \item if~$j<i$, then~$\left.\basicF_\xi^i\pairing_j'(x'_j,y'_j)\right\vert_{t=\xi}=0$;
      \item if~$j>i$, then~$\left.\basicF_\xi^i\pairing_j'(x'_j,y'_j)\right\vert_{t=\xi}=0$.
    \end{itemize}
  \end{lemma}
  We delay the proof of Lemma~\ref{lem:iszero} to after the proof of Proposition~\ref{prop:uniqueness_two}, which we now conclude. Namely,
  given Lemma~\ref{lem:iszero}, the sum on the right-hand side of \eqref{eq:pairing_split} becomes~$\basicF_\xi^i\pairing'_i(x'_i,y'_i)$,
  that is, by definition,~$\pairing_{H'_i}(\wt{x}'_i,\wt{y}'_i)$. This implies that~$\phi$ induces an isometry between~$H_i$ and~$H'_i$.
  This concludes the proof of Proposition~\ref{prop:uniqueness_two} modulo the proof of Lemma~\ref{lem:iszero}.
\end{proof}
\begin{proof}[Proof of Lemma~\ref{lem:iszero}]
  For the first case, note that~$x'_j$,~$y'_j$ are annihilated by~$\basicF_\xi^j$, because they belong to~$M'_j$. The lemma follows immediately
  as~$i>j$.

  For the second case, we know that~$x'_j$ and~$y'_j$ are annihilated by~$\basicF_\xi^i$, so they are divisible by~$\basicF_\xi^{j-i}$ in
 ~$M'_j$. Write~$x'_j=\basicF_\xi^{j-i}x_0$,~$y'_j=\basicF_\xi^{j-i}y_0$. Then
  \[\basicF_\xi^i\pairing_j'(x'_j,y'_j)=\basicF_\xi^{j+(j-i)}\pairing'_j(x_0,y_0).\]
  The latter is zero when evaluated at~$t=\xi$, because~$\pairing'_j(x_0,y_0)$ has a pole at~$t=\xi$ of order at most~$j$.
This concludes the proof of Lemma~\ref{lem:iszero}.
\end{proof}

We can now prove the uniqueness of splitting.
\begin{proof}[Proof of Theorem~\ref{thm:first_splitting}, item (3)]
  If~$\xi\notin\Xi^\F_\ee$, then the isometry class of~$(M,\pairing)$ is determined by the underlying module structure, see Lemma~\ref{lem:underlying}.
  Therefore, it is nothing to prove.

  Given~$\xi\in\Xi^\F_\ee$,  assume $(M,\pairing)$ decomposes in two ways,  i.e. there is an isometry
  \begin{equation}\label{eq:two_equal_one}
    (M,\pairing) \cong \bigoplus_{i\in I}\ee(n_i,\xi,\epsilon_i,\F) \xrightarrow{\phi,\cong} \bigoplus_{j\in J}\ee(m_j,\xi,\delta_j,\F),
  \end{equation}
  where $I$ and $J$ are finite sets of indices. 
  To simplify the notation we assume that $\xi\notin\{-1,1\}$, or $\F=\C$. In the special case
  $\xi\in\{-1,+1\}$, $\F=\R$ the formula \eqref{eq:two_equal_one} looks different, but the proof is exactly the same.
Consider the two following linking forms:
  \begin{align*}
    M_{n,\xi}&=\bigoplus_{i\in I\colon n_i=n}\ee(n_i,\xi,\epsilon_i,\F),\\
    M'_{n,\xi}&=\bigoplus_{j\in J\colon m_j=n}\ee(m_j,\xi,\delta_j,\F).
  \end{align*}
  Let $(H_{n,\xi},\pairing_{H_{n,\xi}})$ and $(H'_{n,\xi},\pairing_{H'_{n,\xi}})$ be the Hermitian forms over $\F$ respectively associated with $M_{n,\xi}$ and $M'_{n,\xi}$. By Proposition~\ref{prop:uniqueness_two}, $\phi$ induces an isometry between 
  $(H_{n,\xi},\pairing_{H_{n,\xi}})$ and $(H'_{n,\xi},\pairing_{H'_{n,\xi}})$. As in the proof of Lemma~\ref{lem:uniqueness_for_ee},  we deduce that for each $\epsilon=\pm 1$, we have
  \[\#\{i\colon n_i=n,\epsilon_i=\epsilon\}=\#\{j\colon m_j=n,\delta_j=\epsilon\}.\]
  But this is precisely the statement that the two decompositions in \eqref{eq:two_equal_one} are equal up to permuting summands.
\end{proof}

\section{Further properties of linking forms}
\label{sec:FurtherProperties}

In this section, we gather some additional results on linking forms. First, in Subsection~\ref{sub:NonSplit} we provide an example showing
that Theorem~\ref{thm:first_splitting}(1) does not necessarily hold for degenerate linking forms.
Subsection~\ref{sec:forms_cyclic} contains an explicit way of finding a decomposition of a linking form over a cyclic module into basic forms. Applications of this result appear in \cite{BCP_Compu}: it is used to calculate twisted signatures of some linear combinations of torus knots.
Subsection~\ref{sub:LocalizationDiagonalization} proves a classification of forms over the local ring~$\mathcal{O}_\xi$. Together with results of Subsection~\ref{sec:loc2},
these techniques are a key tool to prove Proposition~\ref{prop:JumpIsJump} relating two types of signature invariants. Finally, Sections~\ref{sec:Representability} and~\ref{sec:rep2} introduce and discuss the notion of a representable linking form.

\subsection{An example of a non-split linking form}
\label{sub:NonSplit}

As we stated in Theorem~\ref{thm:first_splitting}, non-singular linking forms split as orthogonal sums of forms over cyclic modules, except for summands~$\ff(n,\xi,\R)$ with~$n$ odd and~$\xi\in\{-1,1\}$. The case of singular pairings is more subtle. In this subsection, we give an example indicating that the statement of Theorem~\ref{thm:first_splitting}(1) does not
necessarily hold if the linking form is singular.

\begin{example}
  \label{ex:NonSplit}
  Consider~$\LF$ for~$\F=\R$ or~$\C$ and let~$p$ be an irreducible symmetric polynomial, for instance, one could take~$\F=\R$ and~$p=\basicR_\xi(t)$ for~$\xi\in \Xi^\R_{\ee+}$; see \eqref{eq:big_xi_def_2}.
  We let
 ~$M=\LF/p^5\LF\oplus\LF/p^4\LF$, use~$\alpha~$ and~$\beta$ to denote the generators of the two summands and consider the linking form given by
  \begin{equation}\label{eq:example_pairing}
    \pairing(\alpha,\alpha)=\frac{1}{p},\ \pairing(\beta,\beta)=\frac{1}{p^2},\ \pairing(\alpha,\beta)=\frac{1}{p^3}.
  \end{equation}
  The rest of Subsection~\ref{sub:NonSplit} is devoted to showing that~$\pairing$ is not split. First, observe that~$p^3\pairing(x,y)=0$ for all~$x,y\in M$. Next, note that taking~$x=\alpha$,~$y=\beta$ guarantees the existence of a pair of elements~$x,y\in M$ such that
  \begin{equation}\label{eq:psquarexy}
    p^2\pairing(x,y)\neq 0.
  \end{equation}
  Assume that there is a presentation~$M=\LF/p^5\LF\oplus\LF/p^4\LF$ with generators~$\alpha'$ and~$\beta'$
  such that~$\pairing(\alpha',\beta')=0$. By Classification Theorem~\ref{thm:MainLinkingForm} we may assume that 
 ~$\pairing(\alpha',\alpha')=\frac{\eps_1}{p^k}$ and~$\pairing(\beta',\beta')=\frac{\eps_2}{p^{k'}}$ for some~$k$ and~$k'$ and
 ~$\eps_1,\eps_2=\pm 1$. By~\eqref{eq:example_pairing} we deduce~$k,k'\le 3$. 

  Write the elements~$x$ and~$y$ satisfying \eqref{eq:psquarexy} in the basis~$\alpha',\beta'$, so that~$x=x_1+x_2$,~$y=y_1+y_2$,
  where~$x_1,y_1$ are multiples of~$\alpha'$ and~$x_2,y_2$ are multiples of~$\beta'$. By orthogonality of~$\alpha'$ and~$\beta'$, we have~$\pairing(x,y)=
  \pairing(x_1,y_1)+\pairing(x_2,y_2)$. Since \eqref{eq:psquarexy} holds for~$x$ and~$y$, it must either hold for the pair~$x_1,y_1$, or for
  the pair~$x_2,y_2$. Then at least one of the~$k,k'$ must be equal to~$3$. Without losing generality, we assume that~$k=3$.

  Consequently, writing~$\alpha'=a\alpha+b\beta$ and~$\beta'=a'\alpha+b'\beta$, we obtain
  \begin{equation}
    \label{eq:ContradictionNonSplit}
    \pairing(\alpha',\beta')=\frac{a{a'}^\#}{p}+\frac{b{b'}^\#}{p^2}+\frac{a{b'}^\#+{a'}^\#b}{p^3}.
  \end{equation}
  We now claim that~$p$ divides~$a'$ and~$p$ does not divide~$a,b,b'$. Before proving this claim, observe that this implies that~$p$ does not divide the sum~$a{b'}^\#+{a'}^\#b$. In particular, using~(\ref{eq:ContradictionNonSplit}), we deduce that~$\pairing(\alpha',\beta')$ cannot be zero, contradicting the orthogonality of~$\alpha'$ and~$\beta'$.

  We now prove the claim. First, observe that since we showed that~$k=3$, we have~$p^2\pairing(\alpha',\alpha')\neq~0$. Via \eqref{eq:example_pairing}, this translates into~$p\nmid (ab^\#+a^\# b)$. In particular,~$p$ divides neither~$a$
  nor~$b$. Next, we show that ~$p|a'$. As~$p^4\alpha\neq 0$ and~$p\nmid a$, we have~$p^4\alpha'\neq 0$. Observe that~$p^4\beta'=0$: by the definition of~$M$, we have~$p^4M=\LF/p \LF$, so if~$p^4\beta'\neq 0$, then 
  it would be linearly dependent with~$p^4\alpha'$, contradicting the fact that~$\alpha'$ and~$\beta'$ are generators of~$M$. Consequently, since~$p^4 \alpha \neq 0$, while~$p^4 \beta =0$ and~$p^4 \beta'=0$, the definition of~$\beta'$ implies that~$p|a'$. Finally, we show that~$p\nmid b'$. If we had~$p|b'$, the fact that~$p|a'~$ and the definition of~$\beta'$ would imply that~$\beta'=pz$ for
  some~$z\in M$. This contradicts the fact that~$\beta'$ is a generator. This proves the claim. We have therefore shown that~$(M,\pairing)$ does not split as an orthogonal sum of forms on cyclic modules.
\end{example}

\subsection{Forms over cyclic modules}\label{sec:forms_cyclic}
In this subsection, we state a consequence of Classification Theorem~\ref{thm:MainLinkingForm} which has practical applications
in finding the decomposition~\eqref{eq:splitting}. The results are explicitly used in \cite{BCP_Compu}.
\medbreak

Suppose~$M$ is a cyclic~$\LF$-module, that is~$M=\LF/f$ for some polynomial~$f$.
We use~$1_M$ to denote the generator of~$M$ which is obtained as the image of~$1\in\LF$ under the projection map~$\LF\to\LF/f$. Given a linking form~$\pairing\colon M\times M\to\OF/\LF$, we observe that
\[\pairing(1_M,1_M)=\frac{h}{f}\in\OF/\LF\]
for some~$h$.
It is convenient to think of~$h$ as an element in~$\LF$ defined modulo~$f$. 
Choose one such $h$. Note that $\pairing(1_M,1_M)=\pairing(1_M,1_M)^\#$, hence
$\frac{h^\#}{f^\#}=\frac{h}{f}\in\OF/\LF$. That is to say, $\frac{h^\#}{f^\#}=\frac{h}{f}+g$ for some $g\in\LF$. Applying $\cdot^\#$ to
this identity, we 
obtain $g^\#=-g$. Moreover, with $\wt{h}=h+\frac12fg$ we have that $\frac{\wt{h}^\#}{f^\#}=\frac{\wt{h}}{f}$.
Therefore, on replacing $h$ by $h+\frac12fg$ we may and will assume that $\frac{h^\#}{f^\#}=\frac{h}{f}$. In short, we have choosen a symmetric
representative of $h/f$.

The following result gives a precise description of the isometry type of~$\pairing$ in terms of~$h$ and~$f$.

\begin{proposition}\label{prop:cyclic_classif}
  Let~$\xi\in S^1$ be a root of~$f$ of order~$n>0$ and let~$\basicF_\xi(t)$ be a basic polynomial having root at~$\xi$.
  Consider the restriction~$\pairing_\xi$ of~$\pairing$ to the direct summand~$M_\xi=\LF/\basicF_\xi(t)^n$ of~$M$.
  Then~$(M_\xi,\pairing_\xi)$ is isometric to~$\ee(n,\xi,\eps,\F)$, where
  \begin{itemize}
  \item if~$n$ is even, then~$\eps$ is equal to~$+1$ if~$q=h(\basicF_\xi(t)\basicF_\xi(t)^\#)^{n/2}/f$ is positive near~$\xi$,~$\eps=-1$ if~$q$ is negative near~$\xi$;
  \item if~$n$ is odd, \(\xi \neq \pm 1\), and~$\F=\R$, then~$\eps=+1$, if~$q'=h\basicF_\xi(t)^n/f$ is positive near~$\xi$, otherwise~$\eps=-1$;
  \item if~$n$ is odd and~$\F=\C$, then~$\eps=+1$, if~$q''=h(t-\xi)^{(n+1)/2}(t^{-1}-\ol{\xi})^{(n-1)/2}/f$ is~$\xi$-positive, otherwise~$\eps=-1$.
  \end{itemize}
\end{proposition}

\begin{proof}
  Since the three items are proved in a very similar way, we only give the proof of the first one in the case where~$\F=\R$ (so that~$\basicF_\xi(t)^\#=\basicF_\xi(t)$).
  Write~$q=\frac{h\basicF_\xi(t)^{n}}{f}$, with~$q(\xi)\neq 0$. We have~$\pairing(1_M,1_M)=\frac{q}{\basicF_\xi(t)^{n}}$.

  The natural injection~$M_\xi\to M$ is given 
  by~$1_\xi\mapsto g\cdot 1_M$, where~$g=f/\basicF_\xi(t)^n$.
  Let~$1_\xi$ be the image of~$1$ under the projection~$\LF\to M_\xi$.
  The restricted pairing~$\pairing_\xi$ is now isometric to the following pairing:
  \[ (1_\xi,1_\xi) \mapsto gg^\#\frac{h}{f}=\frac{qgg^\#}{\basicF_\xi(t)^{n}}.\]
  Now~$g$ is invertible in~$\LF/\basicF_\xi(t)^{n}$, so the pairing~$\pairing_\xi$ is isometric to the pairing
  \[ (1_\xi,1_\xi) \mapsto \frac{q}{\basicF_\xi(t)^{n}}.\]
  It follows from Theorem~\ref{thm:elementary_classification} that~$(M_\xi,\pairing_\xi)$ is isometric to~$\ee(n,\xi,\eps,\F)$
  where~$\eps$ is the sign of~$q(\xi)$.
  But this is precisely the statement of the proposition.
\end{proof}

\subsection{Localisation} 
\label{sub:LocalizationDiagonalization}

The techniques used in the proof of Classification Theorem~\ref{thm:MainLinkingForm} can also be applied to study forms on the local ring of germs of holomorphic functions
near~$\xi$, 
that is, forms over the ring ~$\OO_\xi$ of
analytic functions near~$\xi$ for~$\xi\in\C$, i.e. over the ring of convergent power series~$\sum_{i \geq 0} a_i(t-\xi)^i$. 
\medbreak
Throughout this subsection, we assume that~$\F=\C$ and~$\xi\in S^1=\Xi^\C_{\ee}$. Write~$\Omega_\xi$ for the field of fractions of~$\OO_\xi$. As alluded to in Remark~\ref{rem:positive}, the ring~$\OO_\xi$ (and therefore~$\Omega_\xi$) has an involution~$r\mapsto~r^\#$. 

\begin{proposition}\label{prop:classify_cyclic}
  If~$\wh{M}$ is a cyclic~$\OO_\xi$--module, then any linking form~$\wh{\pairing}\colon\wh{M}\times\wh{M}\to\Omega_\xi/\OO_\xi$ is isometric to
  \begin{align*}
    \OO_\xi/(t-\xi)^n\times\OO_\xi/(t-\xi)^n&\to \Omega_\xi/\OO_\xi\\
    x,y&\mapsto \frac{\epsilon r(t) xy^\#}{(t-\xi)^{n}},
  \end{align*}
  where~$n$ are non-negative integers,~$\epsilon=\pm 1$ and~$r(t)$ is such that
  \begin{itemize}
  \item[(a)] if~$n$ is even, then~$\frac{r(t)}{(t-\xi)^{n}}=\frac{1}{((t-\xi)(t^{-1}-\ol{\xi}))^{n/2}}$;
  \item[(b)] if~$n$ is odd, then~$\frac{r(t)}{(t-\xi)^{n}}=\frac{\wt{r}(t)}{(t-\xi)((t-\xi)(t^{-1}-\ol{\xi}))^{(n-1)/2}}$, where~$\wt{r}(t)$ is linear and~$\xi$-positive.
  \end{itemize}
\end{proposition}
\begin{proof}
  The statement resembles Theorem~\ref{thm:elementary_classification}.  As~$\OO_\xi$ is a discrete valuation ring, 
  every cyclic \(\OO_{\xi}\)-module is of the form \(\OO_{\xi} / (t-\xi)^{n}\) and every torsion \(\OO_{\xi}\)-module is a direct sum of cyclic modules. In particular,~$\wh{M}\cong\OO_\xi/(t-\xi)^n$. 
  Suppose that~$n$ is odd (the even case is analogous).

  There exists an analytic function~$q$ such that
  \begin{equation}\label{eq:pairing_in_the_proof}\wh{\pairing}(1,1)=\frac{q}{(t-\xi)((t-\xi)(t^{-1}-\ol{\xi}))^{(n-1)/2}}.\end{equation}
  The symmetry of a pairing implies that~$(t^{-1}-\ol{\xi})q$ is symmetric. In particular,~$q$ is either~$\xi$-positive, or~$\xi$-negative.
  The same proof as in Lemma~\ref{lem:not_iso_ee_C}(c) reveals that the pairings \eqref{eq:pairing_in_the_proof} corresponding to~$\xi$-positive and~$\xi$-negative functions
  are not isomorphic. Suppose~$q_1,q_2\in\OO_\xi$ are~$\xi$-positive. We aim to show that the two pairings \eqref{eq:pairing_in_the_proof}
  defined with~$q_1$ and~$q_2$ are isometric. The quotient~$h=q_1/q_2$ is an analytic function near~$\xi$, and, by 
  Lemma~\ref{lem:residue},~$h(\xi)$ is real positive and~$h^\#=h$. Choose~$g\in\OO_\xi$ such that~$h=g^2$ and~$g(\xi)>0$.
  Then~$(g^\#)^2=h^\#=h$ in~$\OO_\xi$ and~$g^\#(\xi)=g(\xi)>0$. Hence~$g=g^\#$, that is to say,~$q_1=q_2gg^\#$. Multiplication by~$g$
  yields an automorphism of~$M$ transferring a pairing defined with~$q_1$ to a pairing defined with~$q_2$.

  This shows that there are precisely two isometry classes of pairings on~$M$.
\end{proof}

From now on, we use~$\wh{\ee}(n,\epsilon,\xi,\C)$ to denote the linking form described in Proposition~\ref{prop:classify_cyclic}.
Furthermore, given a linking form~$(M,\pairing)$ over~$\LC$, we let~$(\wh{M}_\xi,\wh{\pairing}_\xi)$ denote the linking form 
$$(\wh{M}_\xi,\wh{\pairing}_\xi)=(M,\pairing)\otimes_{\LC}\OO_\xi.$$
In other words, the linking form~$\widehat{\pairing}$ takes values in~$\Omega_\xi/\OO_\xi$. We conclude this subsection with the following observation which we shall use in Subsection~\ref{sec:JumpIsJump}.

\begin{proposition}\label{prop:tensor}
  Suppose that~$(M,\pairing)$ is a linking form over~$\LC$ and choose~$\xi\in S^1=\Xi^{\C}_{\ee}$.
  The~$(t-\xi)$-primary summand of~$(M,\pairing)$ is isometric to~$\oplus_{i\in I} \ee(n_i,\xi,\eps_i,\C)$ if and only if we have the following~isometry: 
  \begin{equation}
    \label{eq:LocalizedGuy}
    (\wh{M}_\xi,\wh{\pairing}_\xi) \cong \bigoplus_{i\in I} \wh{\ee}(n_i,\xi,\eps_i,\C).
  \end{equation}
\end{proposition}
\begin{proof}
  The proposition follows from the following observation: if~$\eta\neq \xi$, then multiplication by~$t-\eta$ is an isomorphism of~$\OO_\xi$
  and, consequently, tensoring by~$\OO_\xi$ kills the~$p$-primary part of~$M$ for any basic polynomial~$p$ coprime with~$t-\xi$.
\end{proof}

\subsection{Representable linking forms}
\label{sec:Representability}
A special class of linking forms is formed by so-called representable linking forms. These linking forms will be frequently used in Section~\ref{sec:Signatures}.

\begin{definition}
  \label{def:RepresentBlanchfield}
  A non-singular linking form~$(M,\pairing)$ over~$R$ is \emph{representable} if there exists a 
  Hermitian matrix~$A$ over~$R$, with $\det(A) \neq 0$,  such that
 ~$(M,\pairing)$ is isometric to~$(R^n/A^T R^n,\pairing_A)$, where the latter linking form is defined by
  \begin{align*}
    \pairing_{A} \colon  R^n /A^TR^n \times R^n/A^T R^n &\to Q/R \\
    ([x],[y]) & \mapsto x^TA^{-1}\makeithash{y}.
  \end{align*}
  In this case, we say that the Hermitian matrix~$A$ \emph{represents} the linking form~$(M,\pairing)$.
\end{definition}

While the terminology appears to be novel, we observe that the concept of representability has already frequently appeared in the literature~\cite{FriedlThesis, BorodzikFriedl0, BorodzikFriedl, BorodzikFriedl2, ConwayFriedlToffoli,ConwayBlanchfield}. However since the conventions sometimes vary, we take a moment for some remarks on the pairing~$\pairing_A$. 

\begin{remark}
  \label{rem:RepresentableWellDefined}
  The pairing~$\pairing_A$ is well-defined since we used~$A^T$ in the presentation of the module, since our pairings are anti-linear in the \emph{second} variable and since, by definition, a matrix~$A$ is Hermitian if~$\makeithashT{A}=A$, see also~\cite[Proposition 4.2]{ConwayBlanchfield}. Note that different conventions were used in~\cite{BorodzikFriedl}, the module was given by~$M=R^n/AR^n$, but the pairing was anti-linear in the \emph{first} variable.
\end{remark}

We will now address the question of which linking forms are representable.
First of all, note that there is a significant difference between representability of real and complex linking forms. 
In the real case it follows from~\cite{BorodzikFriedl2} that any non-singular
linking form (except for the one defined on \((t\pm1)^{n}\)-torsion modules) is representable by a diagonal matrix.
We shall see that this does not hold in the complex case.
Before expanding on the latter case, we state the former result for future reference.

\begin{proposition}\label{prop:diagonalreal}
  For every non-singular linking form over~$\LR$, apart from \(\ff(n,\pm1,\R)\), where \(n\) is odd, there exists a diagonal Hermitian matrix~$A$ over~$\LR$ representing this form.
  For \(\xi = \pm1\), the linking form \(\ff(n,\xi,\R)\), for odd \(n\), is represented by the following \(2 \times 2\) matrix
  \begin{equation}\label{eq:diagonal_H}H =
    \begin{pmatrix}
      0      & (t^{-1}-\xi)^{n} \\
      (t-\xi)^{n} & 0
    \end{pmatrix}.
  \end{equation}
\end{proposition}
\begin{proof}\label{rem:diagonalreal}
  The special case of Proposition~\ref{prop:diagonalreal}, that is if the multiplication by~$t\pm 1$ is an isomorphism i.e. the modules have no~$(t\pm 1)$-torsion is proved in \cite{BorodzikFriedl2}.
  For \(\xi = \pm1\) and even \(n\) this assumption is easy to overcome because the form~$\ee(n,\pm 1,\eps,\R)$ is representable by the~$1 \times 1$ matrix~$((t\mp 2+t^{-1})^{n/2})$.

  For \(\xi = \pm1\) and odd \(n\) we have \(\LR^{2} / H^{T} \cong \module(\xi,n,\R)^{\oplus 2}\).
  Furthermore,
  \[H^{-1} =
    \begin{pmatrix}
      0                    & \frac{1}{(t-\xi)^{m}} \\
      \frac{1}{(t^{-1}-\xi)^{n}} & 0 
    \end{pmatrix},
  \]
  hence the proposition follows.
\end{proof}
\begin{remark}
  We leave to the reader checking that the pairing~$\ff(n,\xi,\R)$ for~$n$ even and~$\xi\in\{-1,1\}$ is not representable by
  a diagonal matrix.
\end{remark}

In order to build some intuition in the complex case, let us start with some examples of representable linking forms.

\begin{proposition}\label{prop:EasyRepresent}
  The following linking forms are representable:
  \begin{enumerate}
  \item the basic linking form~$\ee(n,\eps,\xi,\C)$ for~$n$ even,
  \item the basic linking form ~$\ff(n,\xi,\C)$,
  \item the direct sum~$\ee(n,\eps,\xi,\C)\oplus \ee(n',-\eps,\xi',\C)$, for~$n,n'$ odd and distinct~$\xi,\xi'\in\Xi^\C_{\ee}$.
  \end{enumerate}
\end{proposition}
\begin{proof}
  We begin with the representability of the first two linking forms. For~$n$ even, the linking form~$\ee(n,\eps,\xi,\C)$ is represented by a~$1\times 1$ matrix~$\eps((t-\xi)(t^{-1}-\ol{\xi}))^{n/2}$. The pairing~$\ff(n,\xi,\C)$ is represented by~$((t-\xi)(t^{-1}-\ol{\xi}))^n$.

  We now prove the representability of the third linking form.
  Without loss of generality, suppose that~$n>n'$.
  We need the following simple algebraic lemma, whose proof is deferred after the proof of Proposition~\ref{prop:EasyRepresent}.
  \begin{lemma}\label{lem:used_to_be_quadratic}
    For any~$\xi_1,\xi_2\in S^1$, there exist~$a \in \C$ and~$b \in \R$ such that~$p=at-2b+\ol{a}t^{-1}$ has roots precisely at~$\xi_1$ and~$\xi_2$ 
  \end{lemma}
  Choose~$p$ to be a polynomial with roots at~$\xi$ and~$\xi'$ as in Lemma~\ref{lem:used_to_be_quadratic}.
  Set~$\wt{p}=((t-\xi)(t^{-1}-\ol{\xi}))^{(n-n')/2}p^{n'}$.
  The~$1\times 1$ matrix~$(\wt{p})$ represents a non-singular pairing over the module~$M=B_{\C}(\xi,n) \oplus B_{\C}(\xi',n')$. 
  Using Proposition~\ref{prop:cyclic_classif} (or Proposition~\ref{prop:sumofjumps} below), we see that the pairing is 
  isometric to~$\ee(n,\wt{\eps},\xi)\oplus \ee(n',-\wt{\eps},\xi')$ for some~$\wt{\eps}=\pm 1$.
  If~$\eps\neq\wt{\eps}$, then we replace the matrix~$(\wt{p})$ by~$(-\wt{p})$.
\end{proof}
\begin{proof}[Proof of Lemma~\ref{lem:used_to_be_quadratic}]
  The roots of~$p$ are
  \begin{equation}\label{eq:explicit_roots}
    t_{1,2}=\frac{1}{a}\left(b\pm\sqrt{b^2-|a|^2}\right).
  \end{equation}
  If~$|a|=b$, then~$t_{1,2}$ is a double root equal to~$b/a$, so we obtain a polynomial with the double root on the unit circle.
  If~$|a|^2>b^2$, then~$p$ has two roots outside ~$S^1$. So suppose that~$|a|^2<b^2$.
  Write~$r=b/|a|$ and~$a=|a|e^{i\phi}$. Then by \eqref{eq:explicit_roots}, we obtain
  \[t_{1,2}=e^{-i\phi}(r\pm\sqrt{r-1}).\]
  We observe that~$t_1\ol{t_2}=1$. Moreover, it is clear that any choice of two complex numbers~$t_1,t_2$ such that~$t_1\ol{t_2}=1$ can be realized by a suitable choice of~$a$ and~$b$.
\end{proof}

Next, we give an example of non-representable linking form.  

\begin{proposition}\label{prop:verystupidexample}
  Let~$\xi\in\Xi^\C_\ee$.
  For each positive~$n$, the basic linking form~$\ee(2n+1,\eps,\xi,\C)$ is not representable.
\end{proposition}
\begin{proof}
  Suppose~$A(t)$ is a matrix representing~$\ee(2n+1,\eps,\xi,\C)$.
  Then the determinant of~$A(t)$ is the order of the module~$\LC/(t-\xi)^{2n+1}$, so~$\det A(t)=u(t-\xi)^{2n+1}$ for some unit~$u\in\LC$.
  On the other hand, since~$A(t)$ is Hermitian, its determinant is symmetric.
  The result follows by noting that there is no symmetric polynomial in~$\LC$ that is equal to~$(t-\xi)^{2n+1}u$ for some unit~$u\in\LC$.
\end{proof}
We end this section with an example which will be used in the proof of Proposition~\ref{prop:inversejumps} below.

\begin{proposition}\label{prop:notsostupidexample}
  Given~$\xi \in \Xi^{\C}_\ee$, the linking form~$\ee(1,+1,\xi,\C)\oplus\ee(1,-1,\xi,\C)$ is representable, but not by a diagonal matrix.
\end{proposition}

\begin{proof}
  First of all, if a diagonal matrix~$A$ represents a form~$(M,\pairing)$, and~$A$ has~$b_1(t),\ldots,b_n(t)$ on its diagonal, then the underlying module structure is~$M=\LC/b_1(t)\oplus\ldots\oplus\LC/b_n(t)$ and each of the~$b_i(t)$ is a symmetric polynomial. 

  Assume by way of contradiction that~$\ee(1,+1,\xi,\C)\oplus\ee(1,-1,\xi,\C)$ is representable by a diagonal matrix and let~$M$ be its underlying module.
  We know that~$M \cong B_{\C}(\xi,1) \oplus B_{\C}(\xi,1)$.
  Therefore, we can assume that \(b_{1}(t) \doteq (t-\xi)\) and \(b_{2}(t) \doteq (t-\xi)\) and \(b_{j}(t) \doteq 1\), for \(j \geq 2\).
  But this is impossible since \(b_{1}(t)\) and \(b_{2}(t)\) must be symmetric.
  This yields a contradiction, thus~$\ee(1,+1,\xi,\C)\oplus\ee(1,-1,\xi,\C)$ cannot be represented by a diagonal matrix.

  Next, we show that~$\ee(1,+1,\xi,\C)\oplus\ee(1,-1,\xi,\C)$ is representable.
  Choose four complex numbers~$a,b,c,d$ and consider the matrix
  \begin{equation}\label{eq:trickymatrix}
    A=\begin{pmatrix}
      at^{-1}-\ol{a\xi} & dt^{-1}+c\\
      -\ol{c\xi}t^{-1}-\ol{d\xi} & bt^{-1}-\ol{b\xi}
    \end{pmatrix}.
  \end{equation}
  A direct computation shows that~$(t-\xi)A$ is Hermitian.
  Now we calculate the determinant of~$A$:
  \begin{align*}
    \det A&=(at^{-1}-\ol{a\xi})(bt^{-1}-\ol{b\xi})+(dt^{-1}+c)(\ol{c\xi}t^{-1}+\ol{d\xi})\\
          &=t^{-2}(ab+d\ol{c\xi})-t^{-1}\ol{\xi}(a\ol{b}+b\ol{a}-d\ol{d}-c\ol{c})+(\ol{ab\xi^2}+c\ol{d\xi}).
  \end{align*}
  Suppose~$a,b,c,d$ are such that
  \begin{equation}\label{eq:neweq}
    ab=-d\ol{c\xi} \textrm{ and }2\re(a\ol{b})\neq|c|^2+|d|^2.
  \end{equation}
  It is clear that such~$a,b,c,d\in\C$ exist. One could take~$a=b=c=0$ and~$d\neq 0$, but in the proof of Proposition~\ref{prop:inversejumps}
  below, we will need an example with~$ab\neq 0$. So, take~$a,b,c\neq 0$ and set~$d=\frac{-ab}{c\xi}$. If, accidentally,~$2\re(a\ol{b})=|c|^2+|d|^2$, we can replace~$c,d$ by~$\frac{c}{2}$ and~$2d$. The new values
  satisfy \eqref{eq:neweq}.

  With these conditions, the determinant of~$A$ is equal to the non-zero complex number~$\ol{\xi}(2\re(a\ol{b})-|c|^2-|d|^2)t^{-1}$.
  Since this is a unit in~$\LC$, the matrix~$A$ is invertible, and we now set \[B=(t-\xi)A.\]
  We claim that~$B$ represents~$\ee(1,+1,\xi,\C)\oplus\ee(1,-1,\xi,\C)$.
  To see this, we first note that the Smith normal form of~$B$ is a diagonal matrix with~$((t-\xi),(t-\xi))$ 
  on the diagonal.
  In fact, as~$A$ is invertible, the Smith normal form of~$A$ is the identity matrix, and multiplying a matrix by a polynomial amounts to multiplying the Smith Normal Form by the same polynomial.
  Given this observation, we deduce that~$\LC^2/B\LC^2$ is isomorphic~$\module(1,\xi,\C) \oplus \module(1,\xi,\C)$.
  Using Theorem~\ref{thm:elementary_classification}, the form represented by~$B$ is therefore isometric to~$\ee(1,\eps_1,\xi,\C)\oplus\ee(1,\eps_2,\xi,\C)$ for some~$\eps_1,\eps_2=\pm1$.
  Although the signs~$\eps_1,\eps_2$ can be explicitly calculated, Proposition~\ref{prop:sumofjumps} below immediately implies that~$\eps_1+\eps_2=0$. This concludes the proof of the proposition.
\end{proof}

\subsection{Representability over~$\LC$}\label{sec:rep2}
Generalizing the example given in Proposition~\ref{prop:notsostupidexample}, we can show the following result (whose converse will be proved in Proposition~\ref{prop:sumofjumps}).

\begin{proposition}\label{prop:inversejumps}
  A non-singular linking form~$(M,\pairing)$ over~$\LC$  with decomposition as in~\eqref{eq:splitting}
  is representable if it
  satisfies
  \begin{equation}\label{eq:signatureswithoutsignatures}
    \sum_{i\colon n_i\text{ odd}} \eps_i=0.
  \end{equation}
\end{proposition}
\begin{proof}
  Decompose~$(M,\pairing)$ into basic summands as in \eqref{eq:splitting}. We proceed by induction over the number of summands in this decomposition.
  Suppose that~$(M,\pairing)$ has~$N$ summands and that the result is proved for all forms with fewer than~$N$ summands, that satisfy \eqref{eq:signatureswithoutsignatures}.

  First, suppose~$(M,\pairing)$ contains a basic summand~$(M',\pairing')$ of the type~$\ee(n,\eps,\xi,\C)$ for~$n$ even, or~$\ff(n,\eps,\C)$.
  In this case,~$(M',\pairing')$ is representable by the first two items of Proposition~\ref{prop:EasyRepresent} and we can write~$$(M,\pairing)=(M',\pairing')\oplus(M'',\pairing'')$$  where~$(M'',\pairing'')$ is the sum of other summands.
  Since both~$(M,\pairing)$ and~$(M',\pairing')$ satisfy~\eqref{eq:signatureswithoutsignatures}, so does~$(M'',\pairing'')$.
  We now conclude this case by applying the induction hypothesis to~$(M'',\pairing'')$.

  Next, suppose there is a form~$\ee(2n+1,\xi,+1,\C)$ entering the decomposition~\eqref{eq:splitting} of~$(M,\pairing)$.
  By \eqref{eq:signatureswithoutsignatures}, there must be another form~$\ee(2n'+1,\xi',-1,\C)$ entering the decomposition.
  Write 
 ~$$(M',\pairing')=\ee(2n+1,\xi,+1,\C)\oplus\ee(2n'+1,\xi',-1,\C)$$ and decompose once again~$(M,\pairing)$ as~$(M',\pairing')\oplus (M'',\pairing'')$.
  Since both~$(M',\pairing')$ and~$(M,\pairing)$ satisfy~\eqref{eq:signatureswithoutsignatures}, so does~$(M'',\pairing'')$.
  As~$(M'',\pairing'')$ has~$N-2$ basic summands, the induction assumption applies.
  Therefore, in order to apply the induction hypothesis, we  only need to argue that~$(M',\pairing')$ is representable.
  If~$\xi\neq \xi'$, then the third item of Proposition~\ref{prop:EasyRepresent} ensures that~$(M',\pairing')$ is representable.
  We can therefore assume that~$\xi = \xi'$.
  
  If~$n=n'$, we take the matrix~$A$ of \eqref{eq:trickymatrix} and multiply it by~$(t-\xi)((t-\xi)(t^{-1}-\ol{\xi}))^{n}$ and it represents the desired pairing by the same arguments as in the proof of
  Proposition~\ref{prop:notsostupidexample}. 

  Suppose~$n>n'$ (the other case is analogous).
  Set~$k=n-n'$ and consider the matrix
  \begin{equation}\label{eq:trickymatrix2}
    A'=\begin{pmatrix}
      ((t-\xi)(t^{-1}-\ol{\xi}))^k(at^{-1}-\ol{a\xi}) & (t-\xi)^k(dt^{-1}+c)\\
      -(\ol{c\xi}t^{-1}+\ol{d\xi})(t^{-1}-\ol{\xi})^k & bt^{-1}-\ol{b\xi},
    \end{pmatrix},
  \end{equation}
  where~$a,b,c,d$ are as in \eqref{eq:neweq} and, additionally,~$b$ is not a real number (so that~$\xi$ is not a root of~$(bt^{-1}-\ol{b\xi})$).
  The matrix~$A'$ is such that~$(t-\xi)A'$ is Hermitian and the determinant of~$A'$ is equal to the product of~$((t-\xi)(t^{-1}-\ol{\xi}))^k$ with a unit of~$\LC$.
  We compute the module presented by~$A'$.
  The Smith normal form algorithm allows us to write~$A'$ as
  \begin{equation}\label{eq:aprimsmith}
    A'=X\begin{pmatrix} b_1(t) & 0 \\ 0 & b_2(t) \end{pmatrix} Y,
  \end{equation}
  where~$X$ and~$Y$ are invertible matrices over~$\LC$, and~$b_1(t)b_2(t)=\det A'$ up to a unit of~$\LC$.
  As a shorthand, we write~$\stackrel{.}{=}$ for equality up to multiplication by a unit of~$\LC$.
  Since we have~$\det A'\stackrel{.}{=}((t-\xi)(t^{-1}-\ol{\xi}))^k$, we see that~$b_1\stackrel{.}{=}(t-\xi)^{k_1}$ and~$b_2\stackrel{.}{=}(t-\xi)^{k_2}$ with~$k_1+k_2=2k$.
  From~\eqref{eq:aprimsmith} we deduce that ~$(t-\xi)^{\min(k_1,k_2)}$ divides~$v^TA'w$ for any two vectors~$v,w\in\LC^2$.
  But if~$v=w=(0,1)$, then~\eqref{eq:trickymatrix2} implies that~$v^TA'w=bt^{-1}-\ol{b\xi}$, and the latter is coprime with~$(t-\xi)$ by our choice of~$b$.
  It follows that~$\min(k_1,k_2)=0$. Thus, we may readjust~$X$ and~$Y$ in such a way that~$b_1=(t-\xi)^{2k}$ and~$b_2=1$.

  Since~$(t-\xi)A'$ is Hermitian,~$A''=(t-\xi)((t-\xi)(t^{-1}-\ol{\xi}))^{n'}A'$ is also Hermitian.
  Thus, using \eqref{eq:aprimsmith} and the same arguments as in the proof of Proposition~\ref{prop:notsostupidexample}, we obtain the following isomorphism of~$\LC$-modules:
  \[\LC^2/A''\LC^2\cong\LC/(t-\xi)^{2n+1}\oplus\LC/(t-\xi)^{2n'+1}.\]
  It follows from Classification Theorem~\ref{thm:MainLinkingForm} that~$A''$ represents the linking form
  \[\ee(2n+1,\eps_1,\xi,\C)\oplus\ee(2n'+1,\eps_2,\xi,\C),\]
  where~$\eps_1,\eps_2=\pm 1$.
  By a direct calculation (or using Proposition~\ref{prop:sumofjumps} below), we deduce that~$\eps_1+\eps_2=0$.
  If~$\eps_1=-1$, we replace~$A''$ by~$-A''$, which changes the sign of~$\eps_1$ and~$\eps_2$.
  We conclude that~$A''$ represents~$(M',\pairing')$, so the induction step is accomplished and the proof is concluded.
\end{proof}

\subsection{Diagonalisation over local rings}\label{sec:loc2}

Suppose that~$(M,\pairing)$ is a linking form represented by a Hermitian matrix~$A(t)$.
It is often much easier to deduce some properties of the pairing if~$A(t)$ is diagonal.
While all non-singular linking forms over~$\LR$ are representable by a diagonal matrix (Proposition~\ref{prop:diagonalreal}), this is not the case in general: over~$\LC$ not all representable forms are represented by a diagonal matrix (Proposition~\ref{prop:notsostupidexample}).

We recall the following result, which shows that diagonalisability of~$A$ can be obtained after passing to a local ring.
\begin{lemma}\label{lem:analytic_change}
  Let~$A(t)$ be a Hermitian matrix over~$\LC$ and let~$\xi \in S^1$.
  There exists a matrix~$P(t)$ with entries analytic functions near~$\xi$ such that 
 ~$B(t)=P(t)A(t){P(t)^\#}^T$ is diagonal.
\end{lemma}
\begin{proof}
  Recall that \(\OO_{\xi}\) is a local ring, hence the lemma is a direct consequence of~\cite[Remark II.4.6.5]{knus_quadratic_1991};
  compare also \cite{Adkins,Wimmer}.
\end{proof}

\section{Witt equivalence}
\label{sec:Witt}

While Section~\ref{sec:LinkingFormClassification} dealt with the classification of linking forms over~$\LF$ up to isometry, this section is concerned with the classification up to Witt equivalence. Namely, given a PID with involution, we recall the definition of the Witt group of linking forms over~$R$ (Subsection~\ref{sub:Witt}), deal with polynomial rings (Subsection~\ref{sub:Devissage}) and then specialize to the real and complex cases (Subsections~\ref{sub:WittReal} and~\ref{sub:WittComplex}) all of which are fairly well known. 
Finally, we study morphisms of linking forms in Subsection~\ref{sec:isometricembedding}. The main result of that section, Theorem~\ref{thm:isoprojection},  is used
in the proof of the satellite formula in \cite{BCP_Top}.

\subsection{Witt groups of rings with involution}
\label{sub:Witt}
In this subsection, we review some generalities about Witt groups on rings with involution. 
References include~\cite{Bourrigan, BargeLannesLatourVogel,knus_quadratic_1991,LitherlandCobordism,OrsonThesis, RanickiLocalization, LevineMetabolicHyperbolic}.
\medbreak

Let~$R$ be a PID with involution and let~$Q$ denote its field of fractions. Given a non-degenerate linking form~$(M,\pairing)$ over~$R$, a submodule~$L \subset M$ is \emph{isotropic} (or \emph{sublagrangian}) if~$L \subset L^\perp$ and \emph{metabolic} if~$L=L^\perp$. 
The set of isomorphism classes of $Q/R$-valued non-degenerate linking forms becomes a monoid under the direct sum.
In order to obtain a group structure,  one typically ``quotients out by metabolic forms" or, more formally, declares $(M_1,\lambda_1)$ and $(M_2,\lambda_2)$ to be equivalent if there are metabolic forms $(X_1,\ell_1)$ and $(X_2,\ell_2)$ such that $(M_1,\lambda_1) \oplus (X_1,\ell_1) \cong (M_2,\lambda_2) \oplus (X_2,\ell_2).$ 
Over a PID, one can alternatively proceed as follows.

\begin{definition}
  \label{def:metabolic_and_so_on}
  The \emph{Witt group of linking forms}, denoted~$W(Q,R)$, consists of the quotient of the monoid of isomorphism classes of non-degenerate linking forms by the submonoid of isomorphism classes of metabolic linking forms. 
  Two non-degenerate linking forms~$(M,\pairing)$ and~$(M',\pairing')$ are called \emph{Witt equivalent} if they represent the same element in~$W(Q,R)$.
\end{definition}

The Witt group of linking forms is known to be an abelian group under direct sum, where the inverse of the class~$[(M,\pairing)]$ is represented by~$(M,-\pairing)$. Recall that since the ring~$R$ is a PID, non-degenerate linking forms over a torsion~$R$-module are in fact non-singular~\cite[Lemma 3.24]{Hillman}. Regarding the notation~$W(Q,R)$ we chose to follow \cite[Appendix A.1]{BargeLannesLatourVogel}.
Given a non-singular linking form~$(M,\pairing)$ and an isotropic submodule~$L \subset M$, we define a form~$\pairing_L$ on~$L^\perp/L$ by the formula
$\pairing_L([x],[y])=\pairing(x,y)$, where~$[x],[y]\in L^\perp/L$ are classes of elements~$x,y\in L^\perp$. It can be checked that~$\pairing_L$
is well-defined and in fact, gives a non-singular linking form~$(L^\perp/L,\pairing_L)$.
\begin{definition}\label{def:sublagrangian_reduction}
  The non-singular linking form~$(L^\perp / L,\pairing_L)$ is said to be obtained by \emph{sublagrangian reduction} on~$L$. 
\end{definition}
We refer to~\cite[Corollary A.12]{BargeLannesLatourVogel} and~\cite[Section II.A.4]{Bourrigan} for the proof of the following proposition.

\begin{proposition}
  \label{prop:Reduction}
  Given an isotropic subspace~$L$ of a non-singular linking form~$(M,\pairing)$, the non-singular linking forms~$(M,\pairing)$ and ~$(L^\perp/L,\pairing_L)$ are Witt equivalent.
\end{proposition}

From now on, we shall restrict ourselves to our ring of interest, namely to the ring~$\LF$ of Laurent polynomials, where~$\F$ is either~$\R$ or~$\C$.

\subsection{Witt group of Laurent polynomial rings and d\'evissage}
\label{sub:Devissage}

This subsection has two goals. Firstly, we recall how the computation of Witt groups is simplified by considering the primary decomposition.
Secondly, we recall the process known as \emph{d\'evissage}.
References for this section include \cite{MilnorHusemoller, Lam, Bourrigan, OrsonThesis, LevineMetabolicHyperbolic}. 
\medbreak
We start by introducing some additional terminology. A Laurent polynomial~$p$ is \emph{weakly symmetric} if~$p \stackrel{.}{=} p^\#$.
Given an irreducible polynomial~$p$ in~$\LF$, we define \(W(\OF,\LF,p)\) to be the Witt group of linking forms \((M,\pairing)\), where 
$M$ is annihilated by some power of~$pp^{\#}$.

\begin{proposition}
  \label{prop:Primary}
  Let~$\mathcal{S}$ denote the set of weakly symmetric irreducible polynomials over \(\LF\).
  There is a canonical isomorphism
  \[ W(\OF,\LF)\cong\bigoplus_{p \in \mathcal{S}} W(\OF,\LF,p).\]
\end{proposition}
\begin{proof}
Given a polynomial $p \in \LF$, we write $M_p=\lbrace x \in M  \mid p^nx=0 \text{ for some } n \geq 0 \rbrace$ for the $p$-primary summand of $M$.
  It follows immediately from the definition that if
$p,q^\#$ do not agree up to multiplication by a unit of $\LF$,  then~$M_p$ and~$M_q$ are orthogonal. 
Likewise, if~$p$ is not weakly symmetric, then the linking form restricted to~$M_p\oplus M_{p^\#}$ is metabolic, in fact~$M_p$ and~$M_{p^\#}$ are two metabolizers. 
\end{proof}

In order to describe d\'evissage, we recall some generalities on Witt groups of Hermitian forms.
Let~$R$ be a ring with involution and let~$u$ be an element of~$R$ which satisfies~$u\makeithash{u}=1$.
Given a projective~$R$-module~$H$, a sesquilinear form~$\alpha \colon H \times H \to R$ is \emph{$u$-Hermitian} if~$\alpha(y,x)=u \alpha(x,y)^\#$.
We use~$W_u(R)$ to denote the Witt group of non-singular~$u$-Hermitian forms. If~$u=1$, then we write~$W(R)$ instead of~$W_1(R)$.

\begin{remark}
  \label{rem:ElementuWitt}
  Assume that~$\F$ is a field whose characteristic is different from~$2$. It can be checked that~$W_u(\F)$ is trivial if~$u=-1$ and the involution is trivial, and is isomorphic to~$W(\F)$ otherwise~\cite[Proposition A.1]{LitherlandCobordism} and~\cite[Chapter II.A.3]{Bourrigan}.
  Nevertheless, as we shall see below, keeping track of this extra unit often proves useful.
\end{remark}

The next proposition (which is mentioned in~\cite[top of page 349]{LitherlandCobordism}) is an instance of \emph{d\'evissage}; a proof can be found in~\cite[Chapter II.B.2]{Bourrigan} or~\cite[Proposition
2.3.14]{OrsonThesis}.

\begin{proposition}
  \label{prop:Devissage}
  Given an irreducible weakly symmetric polynomial~$p$, there is an isomorphism
  \[ W(\OF,\LF,p) \cong W_{\wt{u}}(\LF/p),\]
  where~$\wt{u} \in \LF/p$ is the reduction mod~$p$ of the element~$u\in\LF$ such that~$p=u\makeithash{p}$.
\end{proposition}

Next, we specialize our study to the case where~$\F$ is either~$\R$ or~$\C$.

\subsection{Real linking forms up to Witt equivalence.}
\label{sub:WittReal}

This subsection has two goals: firstly to discuss the structure of the group~$W(\OR,\LR)$ and secondly to prove the following result.

\begin{theorem}
  \label{thm:WittClassificationReal}
  Recall that~$\Xi^{\R}_{\ee+}=\lbrace z \in S^1 \ | \ \iim(z)>0 \rbrace$.
  Any real non-singular linking form is Witt equivalent to a direct sum of the type
 ~$$ \bigoplus_{\substack{ n_i \text{ odd, }\eps_i=\pm 1
      \\ \xi_i \in \Xi^{\R}_{\ee+}, \  i\in I 
    }}\ee(n_i,\eps_i,\xi_i,\R).$$
\end{theorem}

Here recall that
~$\ee(n_i,\eps_i,\xi_i,\R)$ denotes the basic linking form described in Example~\ref{BasicPairingXiS1Real} and
$\Xi^\R_{\ee+}$ is the upper half-circle, see \eqref{eq:big_xi_def_2}.
The proof of Theorem~\ref{thm:WittClassificationReal} is based on the two following lemmas which show that the basic linking forms which do not appear in the statement of Theorem~\ref{thm:WittClassificationReal} are metabolic.

\begin{lemma}\label{lem:xiless1}
  Given~$\xi \in \Xi^\R_\ff$, then the Witt group~$W(\OR,\LR,\basicR_\xi(t))$ vanishes.
\end{lemma}

\begin{proof}
  If~$\xi\in\Xi^\R_\ff$, then either~$|\xi|\in(0,1)$, or~$\xi\in\{-1,1\}$.
  Consider the first case. The polynomial~$\basicR_\xi(t)$ decomposes as a product of two weakly symmetric polynomials:~$\basicR_\xi(t)=pp^\#$ where~$p=(t-\xi)(t-\ol{\xi})$. If~$M$ is a~$\basicR_\xi(t)$-torsion module then~$M_p$ and~$M_{p^\#}$ are two metabolizers for the linking form. 

  Consider now~$\xi\in\{-1,1\}$.
  Observe, that for odd \(n\), the linking form \(\ff(n,\xi,\R)\) is hyperbolic, hence it is trivial in the Witt group.
  For even \(n\), one can check that the submodule generated by~$(t-\xi)^{n/2}\cdot [1]$ (where~$[1]$ is the class of~$1$ in~$\LR/(t-\xi)^n$) is a metabolizer for~$\ee(n,\eps,\xi,\R)$.
  This concludes the proof of the claim.
\end{proof}

From a more abstract point of view, the second point of Lemma~\ref{lem:xiless1} can be proved as follows: Thanks to Proposition~\ref{prop:Devissage}, we know that~$W(\OR,\LR,\basicR_\xi(t))$ is isomorphic to~$W_{-1}(\LR/\basicR_\xi(t))$. Since the latter group is isomorphic to~$W_{-1}(\R)$ (with the trivial involution), the result follows from Remark~\ref{rem:ElementuWitt}.
Recall that $\Xi^\R_{\ee+}$ is the set of complex numbers on the unit circle with positive imaginary part.

\begin{lemma}\label{lem:LRxi}
  If~$\xi\in\Xi^\R_{\ee+}$, then there is an isomorphism~$W(\OR,\LR,\basicR_\xi(t))\cong \Z$ under which the basic pairing~$\ee(n,\eps,\xi,\R)$ is mapped to~$\eps$ if~$n$ is odd and to~$0$ if~$n$ is even.
\end{lemma}

\begin{proof}
  As~$\xi$ is not real, we deduce that~$\LR/\basicR_\xi(t)$ is isomorphic to~$\C$. Using Proposition~\ref{prop:Devissage} and Remark~\ref{rem:ElementuWitt}, it follows that~$W(\OR,\LR,\basicR_\xi(t))$ is isomorphic to~$W(\C) \cong\Z$ (where this latter isomorphism is given by the signature).

  We now move on to the second statement.
  First, we assume that~$n=2m$ is even. Set~$H:=\LR/\basicR_\xi(t)^{2m}$.
  Considering~$L=\basicR_\xi(t)^m H$, using the symmetry of~$\basicR_\xi(t)$ and the non-singularity of the pairing, we observe that~$L=L^\perp$ is a metabolizer for~$\ee(2m,\eps,\xi,\R)$.
  Next assume that~$n=2m+1$ is odd, set~$H:=\LR/\basicR_\xi(t)^{2m+1}$ and~$\pairing=e(2m+1,\eps,\xi,\R)$.
  Considering~$L=\basicR_\xi(t)^{m+1}H$, one sees that~$L^\perp$ is equal to~$\basicR_\xi(t)^mH$.
  It follows that multiplication by~$\basicR_\xi(t)^m$ induces an isometry between the sublagrangian reduction~$(L^\perp/L,\pairing_L)$ and~$e(1,\eps,\xi,\R)$. Proposition~\ref{prop:Reduction} now implies that~$\ee(2m+1,\eps,\xi,\R)$ is Witt equivalent to~$\ee(1,\eps,\xi,\R)$.
  
  To conclude the proof, it only remains to show that the basic linking form~$\ee(1,\eps,\xi,\R)$ is mapped to~$\eps$ under the aforementioned isomorphism.
  This is fairly immediate since under the isomorphisms~$W(\OR,\LR,\basicR_\xi(t)) \cong W_{\widetilde{u}}(\LR/\basicR_\xi(t)) \cong W(\C)$, the image of the linking form~$\ee(1,\eps,\xi,\R)$ is congruent to the  Hermitian form~$(x,y) \mapsto \eps xy^\#$.
\end{proof}

Combining the previous lemmas, we can prove Theorem~\ref{thm:WittClassificationReal} which states that any real non-singular linking form is Witt equivalent to a direct sum of~$\ee(n_i,\eps_i,\xi_i,\R)$, where each~$n_i$ is odd and each~$\xi_i$ lies in~$\Xi^{\R}_{\ee+}$.

\begin{proof}[Proof of Theorem~\ref{thm:WittClassificationReal}]
  Let~$(M,\pairing)$ be a real non-singular linking form.
  The Classification Theorem~\ref{thm:MainLinkingForm} implies that~$(M,\pairing)$ must be isometric to a direct sum of basic forms:
 ~$$\bigoplus_{\substack{ n_i,\eps_i,\xi_i\\ i\in I}}\ee(n_i,\eps_i,\xi_i,\R)\oplus
  \bigoplus_{\substack{n_j,\xi_j\\ j\in J}}\ff(n_j,\xi_j,\R).$$
  Using Lemma~\ref{lem:xiless1}, we know that the basic forms~$\ff(n,\xi,\R)$ and~$\ee(n,\eps,\pm 1,\R)$ are Witt trivial Furthermore, Lemma~\ref{lem:LRxi} implies that~$\ee(n,\eps,\xi,\R)$ is metabolic if~$n$ is even. This concludes the proof of the theorem.
\end{proof}

We conclude this subsection by discussing the structure of the group~$W(\OR,\LR)$, see~\cite[Section 5]{MilnorInfiniteCyclic},~\cite[Section 2.10]{Hillman} and~\cite[Example 2.3.24]{OrsonThesis} for closely related discussions. 
Combining the decomposition described in Proposition~\ref{prop:Primary} with the results of Lemmas~\ref{lem:xiless1} and~\ref{lem:LRxi}, we obtain

\begin{equation}
  \label{eq:WittReal}
  W(\OR,\LR) \cong \bigoplus_{\xi \in \Xi^{\R}_{\ee+}} W(\OR,\LR,\basicR_\xi(t)) \cong \bigoplus_{\xi \in \Xi^{\R}_{\ee+}} \Z.
\end{equation}

In more detail, Proposition~\ref{prop:Primary} implies that~$W(\OR,\LR)$ is isomorphic to the direct sum of the~$W(\OR,\LR,\basicR_\xi(t))$, where~$\basicR_\xi(t)$ ranges over the basic polynomials.  Lemmas~\ref{lem:xiless1} and~\ref{lem:LRxi} compute each of these groups, whence the announced isomorphism.

\subsection{Complex linking forms up to Witt equivalence}
\label{sub:WittComplex}

Just as in the real case, this subsection has two goals: firstly to discuss the structure of the group~$W(\OC,\LC)$ and secondly to prove the following result:

\begin{theorem}
  \label{thm:WittClassificationComplex}
  Any complex non-singular linking form is Witt equivalent to a direct sum:
 ~$$ \bigoplus_{\substack{ n_i \text{ odd, }\eps_i=\pm 1 \\ \xi_i \in \Xi^\C_\ee, \ i\in I}}\ee(n_i,\eps_i,\xi_i,\C).$$
\end{theorem}

The proof of Theorem~\ref{thm:WittClassificationComplex} proceeds as in the real case. Namely, we deal with one basic linking form at a time.

\begin{lemma}\label{lem:xiless1a}
  The following statements hold:
  \begin{enumerate}
  \item If~$\xi\in\Xi^\C_\ff$, then the Witt group~$W(\OC,\LC,\basicC_\xi)$ vanishes.
  \item If~$\xi\in\Xi^\C_\ee$, then there is an isomorphism~$W(\OC,\LC,\basicC_\xi)\cong \Z$. Furthermore, this isomorphism maps the basic pairing~$\ee(n,\eps,\xi,\C)$ to~$\eps$ if~$n$ is odd and to~$0$ if~$n$ is even.
  \end{enumerate}
\end{lemma}
\begin{proof}
  The proof of the first assertion is the same as in the real case, see Lemma~\ref{lem:xiless1}, and will be omitted.
  The proof the second assertion is analogous to the proof of Lemma~\ref{lem:LRxi}.
  Namely, using d\'evissage (recall Proposition~\ref{prop:Devissage}), we see that~$W(\OC,\LC,\basicC_\xi)$ is isomorphic to~$W_{-\xi^2}(\LC/\basicC_\xi)$.
  But now, since~$\basicC_\xi=t-\xi$ and~$\LC/(t-\xi)$ is isomorphic to~$\C$, Remark~\ref{rem:ElementuWitt} implies that the latter group is isomorphic to~$\Z$, as desired.
  Finally, as in the proof of Lemma~\ref{lem:LRxi} we show that under this isomorphism~$\ee(n,\eps,\xi,\C)$ is mapped to~$0$ if~$n$ is even and to~$\eps$ if~$n$ is odd.
\end{proof}
\begin{proof}[Proof of Theorem~\ref{thm:WittClassificationComplex}]
  Let~$(M,\pairing)$ be a complex non-singular linking form.
  The Classification Theorem~\ref{thm:MainLinkingForm} implies that~$(M,\pairing)$ form must be isometric to a direct sum of basic forms:
 ~$$\bigoplus_{\substack{ n_i,\eps_i,\xi_i\\ i\in I}}\ee(n_i,\eps_i,\xi_i,\C)\oplus
  \bigoplus_{\substack{n_j,\xi_j\\ j\in J}}\ff(n_j,\xi_j,\C).$$
  Using the first point of Lemma~\ref{lem:xiless1a}, we know that the basic forms~$\ff(n,\xi,\C)$ are metabolic. The second point of Lemma~\ref{lem:xiless1a} implies that~$\ee(n,\eps,\xi,\C)$ is metabolic if~$n$ is even. This concludes the proof of the theorem.
\end{proof}

We conclude this subsection by discussing the structure of the group~$W(\OC,\LC)$ (note that this result has appeared in~\cite[Appendix A]{LitherlandCobordism},~\cite[Chapter II.C]{Bourrigan} and~\cite[Example 2.3.24]{OrsonThesis}). Namely,  combining the decomposition described in Proposition~\ref{prop:Primary} with Lemma~\ref{lem:xiless1a}, we obtain

\begin{equation}
  \label{eq:WittComplex}
  W(\OC,\LC) \cong \bigoplus_{\xi \in \Xi^{\C}_{\ee}} W(\OC,\LC,t-\xi) \cong \bigoplus_{\xi \in \Xi^{\C}_\ee} \Z.
\end{equation}

In more details, Proposition~\ref{prop:Primary} implies that~$W(\OC,\LC)$ is isomorphic to the direct sum of the~$W(\OC,\LC,t-\xi)$, where~$\xi$ ranges over~$\Xi^{\C}_{\ee}=S^1$.
Lemma~\ref{lem:xiless1a} computes each of these groups, whence the announced isomorphism.

\subsection{Witt equivalence and representability}\label{sub:witt_rep}
In this subsection, we briefly outline how representability (see Section~\ref{sec:Representability}) 
fits into the general framework of Witt groups. 
\medbreak
Instead of describing the integrality of the localisation exact sequence of Witt groups/symmetric L-groups (see e.g.~\cite{RanickiExact, RanickiLocalization}), we focus on the part of the theory that is  relevant to representability. Namely, we describe the map 
$$\partial \colon W(\F(t)) \to W(\F(t),\LF).$$
A \emph{lattice} for an~$\F(t)$-Hermitian form~$(E,\alpha)$ is a free~$\LF$-submodule~$P \subset E$ which satisfies~$\alpha(x,y) \in \LF$ for all~$x,y$ in~$P$, i.e. such that~$\alpha$ restricts to a well-defined~$\LF$-homomorphism~$P \to \operatorname{Hom}_{\LF}(P,\LF)^\#$. 
Given such a lattice, one can consider the \emph{dual lattice}
$$ P_d:=\lbrace x \in E \ | \ \alpha(x,p) \in P \text{ for all } p \in P  \rbrace.~$$
Observe that the~$\F(t)$-isomorphism~$x \mapsto \alpha(x,-)$ restricts to a well-defined~$\LF$-isomorphism~$P_d \to \operatorname{Hom}_{\LF}(P,\LF)^\#$. Since~$P$ is free,~$\operatorname{Hom}_{\LF}(P,\LF)^\#$ is torsion-free. Since~$\LF$ is a PID, this implies that~$P_d$ is in fact free. Since~$P$ is a lattice, the canonical inclusion~$P \subset E$ induces a well-defined inclusion~$P \subset P_d$, and it can be checked that~$P_d / P$ is~$\LF$-torsion. One can then verify that the following assignment gives rise to a well-defined linking form on~$P_d /P$:
$$ \partial \alpha ([x],[y]) :=\alpha(x,y).$$
It can also be checked that~$\partial$ descends to a well-defined map on the level of Witt groups. In fact, as a Witt class~$\alpha$ in~$W(\F(t))$ can always be represented by a Hermitian matrix~$A$ with coefficients in~$\LF$ and non-zero determinant, a computation shows that~$\partial \alpha$ is isometric to the linking form represented by~$A$. In particular, a linking form~$(M,\pairing)$ is Witt equivalent to a representable linking form if and only if~$[(M,\pairing)]$ lies in the image of~$\partial$. 

\begin{remark}
  \label{rem:Litherland}
  Using Proposition~\ref{prop:Primary}, one can decompose~$W(\F(t),\LF)$ into the direct sum of~$W^0(\F(t),\LF):=\bigoplus_{p \in \mathcal{S} \setminus \lbrace t-1 \rbrace} W(\F(t),\LF,p)$ and~$W(\F(t),\LF,t-1)$.
  It has been claimed that the map~$\partial$ restricts a surjection~$W(\F(t)) \to W^0(\F(t),\LF)$~\cite[Theorem A.2]{LitherlandCobordism}.
  While this result is correct for~$\F=\R$ (since linking forms over~$\R[t^{\pm 1}]$ are always representable) it is incorrect over~$\C[t^{\pm 1}]$ since for any~$\xi \in \C$, the basic linking form
 ~$\ee(1,1,\xi,\C)$ is not Witt equivalent to a representable form.
  This can either be seen by reasoning as in Proposition~\ref{prop:verystupidexample} (and using that if~$(M,\pairing)$ is metabolic, then the order of~$M$ is a norm) or by using Proposition~\ref{prop:sumofjumps} below.
\end{remark}

\subsection{Forms restricted to submodules and Witt equivalence}\label{sec:isometricembedding}
In this subsection, we are concerned with certain technical result which plays a key role in~\cite{BCP_Top}.

\begin{theorem}\label{thm:isoprojection}
  Let~$(M',\pairing')$ and~$(M'',\pairing'')$ be two non-singular linking forms over~$\LF$, let~$M$ be a~$\LF$-module and let 
 ~$\iota\colon M\to M'$ be a monomorphism.
  Write~$\pairing$ for the restriction of \(\pairing'\) to \(M\).
  Suppose that~$\pi\colon(M,\pairing)\to(M'',\pairing'')$ is a surjective morphism of linking forms and
  set~$L:=\ker(\pi)$ and~$L':=\iota(L)$.
  Then~$\ord(M)\ord(M)^\#$ divides~$\ord(M')\ord(M'')$.
  Moreover, if
  \begin{equation}\label{eq:order}
    \ord(M)\ord(M)^\#\doteq \ord(M'')\ord(M'),
  \end{equation}
  then the following statements hold:
  \begin{enumerate}
  \item~$\iota(M)={L'}^{\perp}$;
  \item The linking form~$(M'',\pairing'')$ is isometric to the sublagrangian reduction of~$M'$ with respect to~$L'$. In particular~$(M',\pairing')$ and~$(M'',\pairing'')$ are Witt equivalent.
  \end{enumerate}
\end{theorem}
\begin{proof}
  Firstly, observe that since \(L\) is the kernel of \(\pi\), it follows that for any \(x \in L\) and any \(y \in M\), \(\pairing(x,y) = 0\).
  Consequently, for any \(x' \in L'\) and any \(y' \in \iota(M)\) we have \(\pairing'(x',y')=0\), therefore \(\iota(M) \subset (L')^{\perp}\).
  Let \(X = (L')^{\perp}\).
  Observe that there is an exact sequence
  \[0 \to X \to M' \to L^{\#} \to 0,\]
  therefore, \(\ord(M') = \ord((L')^{\#}) \ord(X) = \ord(L')^{\#} \ord(X)\).
  Furthermore, \(\ord(X) = \ord(X/M) \ord(M)\) and \(\ord(M) = \ord(M'') \ord(L)\).
  Consequently,
  \[\ord(M') \ord(M'') \doteq \ord(X/M) \ord(M) \underbrace{\ord(M'') \ord(L)^{\#}}_{\doteq \ord(M)^{\#}} \doteq \ord(X/M) \ord(M) \ord(M)^{\#},\]
  proving the first statement of the theorem.
  Note that the second equality follows from the fact that \(\ord(M'') \doteq \ord(M'')^{\#}\).

  Observe that if \(\ord(M') \ord(M'') \doteq \ord(M) \ord(M)^{\#}\), it follows that \((L')^{\perp} = \iota(M)\), therefore, \(\pi\) induces an isometry between \((M'',\pairing'')\) and the sublagrangian reduction of \((M',\pairing')\) with respect to \(L'\)
\end{proof}

\section{Signatures of linking forms}
\label{sec:Signatures}

In Subsection~\ref{sub:SignatureJump} we give a definition of the local signature jump, which echoes Milnor's definition~\cite{MilnorInfiniteCyclic}.
These signature jumps are used in Subsection~\ref{sub:sig_function} to define the signature function. 
In Subsection~\ref{sec:sigrep} we study the signatures of forms that are
representable and give a formula for the signature function of a representable form in terms of the signatures of a  matrix that represents it. As a consequence, we give an
explicit obstruction to representability. Subsection~\ref{sec:JumpIsJump} contains the proof of Proposition~\ref{prop:JumpIsJump}, which is the main technical result in this section. 

\subsection{Signature jumps}
\label{sub:SignatureJump}

In this subsection, we introduce signature jumps and explain how they obstruct a linking form from being metabolic.

\medbreak

Recall from Definition~\ref{def:hodge_number} that the \emph{Hodge number}~$\hodgep(n,\eps,\xi,\F)$ of a non-singular linking form~$(M,\pairing)$ is the number of times~$\ee(n,\eps,\xi,\F)$ enters its decomposition. Recall furthermore from Theorems~\ref{thm:WittClassificationReal} and~\ref{thm:WittClassificationComplex} that every non-singular linking form is Witt equivalent to a direct sum of the forms~$\ee(n,\eps,\xi,\F)$ in which each~$n$ is odd. Motivated by this result, we introduce the following terminology.

\begin{definition}\label{def:sigjump}
  Let~$(M,\pairing)$ be a non-singular linking form over~$\LF$. If the linking form is real, then the \emph{signature jump} of~$(M,\pairing)$ at~$\xi \in S^1_+$ is defined as
  \[\ds_{(M,\pairing)}(\xi)=\sum_{\substack{n \textrm{ odd}\\ \eps=\pm 1}} \eps \hodgep(n,\eps,\xi,\R).\]
  The signature jump at~$\xi \in S^1_-$ is defined as~$\ds_{(M,\pairing)}(\xi)=-\ds_{(M,\pairing)}(\ol{\xi})$.
  If the linking form is complex, then the \emph{signature jump} of~$(M,\pairing)$ at~$\xi \in S^1$
  is defined as
  \[\ds_{(M,\pairing)}(\xi)=-\sum_{\substack{n \textrm{ odd}\\ \eps=\pm 1}} \eps \hodgep(n,\eps,\xi,\C).\]
\end{definition}
Here we use $S^1_{\pm}$ to denote the upper, respectively lower, semicircle:  $S^1_{\pm}=\{z\, |\, |z|=1,\ \pm\im z>0\}$.

Note that a linking form only admits a finite number of non-zero signature jumps. On the other hand, the appearance of the minus sign in Definition~\ref{def:sigjump} will become clear in Proposition~\ref{prop:JumpIsJump}. 

Here is an example of Definition~\ref{def:sigjump}. We refer to \cite{BCP_Compu} for more involved computations.

\begin{example}
  \label{ex:Trefoil}
  Consider the non-singular linking form given by~$(x,y) \mapsto \frac{xy^\#}{t-1+t^{-1}}$. Observe that~$\omega=e^{2\pi i/6}$ and its complex conjugate are the only roots of~$t-1+t^{-1}$. It follows that, as a real linking form,~$(M,\pairing)$ is isometric to~$\ee(1,1,\omega,\R)$. We therefore deduce that
 ~$$\ds_{(M,\pairing)}(\xi) =
  \begin{cases}
    -1 & \text{ if } \xi=\overline{\omega}, \\
    1 & \text{ if } \xi=\omega, \\
    0 &\text{ otherwise.}
  \end{cases}
 ~$$ 
\end{example}

Next, we observe that the signature jumps obstruct a linking form from being metabolic.

\begin{theorem}
  \label{thm:metabolic_signature}
  When~$\F=\R,\C$, the signature jumps define homomorphisms from~$W(\OF,\LF)$ to~$\Z$. Moreover, a linking form over~$\LF$ is metabolic if and only if all its signature jumps vanish.
\end{theorem}
\begin{proof}
  The first assertion follows by combining Definition~\ref{sub:SignatureJump} with Theorems~\ref{thm:WittClassificationReal} and~\ref{thm:WittClassificationComplex}: these results show that a basic linking form contributes to the signature jump if and only if it is not Witt trivial.
  To prove the second assertion, we use Theorems~\ref{thm:WittClassificationReal} and~\ref{thm:WittClassificationComplex} to assume without loss of generality that the linking form is (Witt equivalent to) a direct sum of~$\ee(n,\eps,\xi,\F)$, where each~$n$ is odd. Under the isomorphism 
  \begin{equation}\label{eq:displayed_isomorphism}
    W(\OF,\LF,\basicR_\xi(t)) \cong \Z  
  \end{equation}
  described in Lemmas~\ref{lem:xiless1} and~\ref{lem:xiless1a}, Definition~\ref{def:sigjump} implies that the aforementioned direct sum of basic forms is Witt trivial if and only if the corresponding signature jumps vanish. 
\end{proof}

\subsection{The signature function and the average signature function}\label{sub:sig_function}
In this subsection, we use the signature jumps in order to define the signature function of a non-singular linking form.
\medbreak 
In a nutshell, a signature function of a linking form~$(M,\pairing)$ can be defined as follows: fix the value of the function at a point~$\xi_0\in S^1$ and define the function at~$\xi \in S^1 \setminus \lbrace \xi_0 \rbrace$ by adding up the signature jumps between~$\xi_0$ and~$\xi$ (going anticlockwise). For reasons discussed below, we shall fix the value of the signature function to be~$\ds_{(M,\pairing)}(1)$ as~$\xi$ approaches~$1$ going clockwise along~$S^1$. 

With this intuition in mind, we now give the precise definition of the signature function.

\begin{definition}\label{def:sig_func}
  Suppose~$(M,\pairing)$ is a non-singular linking form. The \emph{signature function} of~$(M,\pairing)$ is the map~$\sigma_{(M,\pairing)} \colon S^1 \to \Z$ whose value at~$\xi_1=e^{2\pi i\theta_1}$ is defined as 
  \begin{equation}\label{eq:def_sig_func}
    \sigma_{(M,\pairing)}(\xi_1)=\sum_{\tau\in(0,\theta_1)} 2\ds_{(M,\pairing)}(e^{2\pi i \tau})-\sum_{\substack{\eps=\pm 1\\n\textrm{ even}}} \eps \hodgep(n,\eps,\xi_1,\F)+\ds_{(M,\pairing)}(\xi_1)+\ds_{(M,\pairing)}(1).
  \end{equation}
\end{definition}

Note that this signature function differs from the fomula in~\cite[Proposition 4.14]{BorodzikNemethi} by an overall constant. In our case, we have~$\lim_{\theta\to 0^+}\sigma_{(M,\pairing)}(e^{i\theta})=\ds_{(M,\pairing)}(1)$ while in~\cite{BorodzikNemethi}, this is not always the case. 
We now describe some basic properties of the signature function.

\begin{proposition}
  \label{prop:BasicProperties}
  Let~$(M,\pairing)$ be a non-singular linking form over~$\LF$ and let~$\Delta_M(t)$ denote the order of the~$\LF$-module~$M$.
  \begin{enumerate}
  \item The signature function is locally constant on~$S^1 \setminus \lbrace \xi \in S^1 \ | \ \Delta_M(\xi)=0 \rbrace$.
  \item If~$(M,\pairing)$ is a real linking form, then we have~$\sigma_{(M,\pairing)}(\overline{\xi})=\sigma_{(M,\pairing)}(\xi)$. 
  \end{enumerate}
\end{proposition}
\begin{proof}
  By Remark~\ref{rem:OrderAndClassification}, the zeros of~$\Delta_M(t)$ are precisely the values of~$\xi$ which enter the decomposition of~$(M,\pairing)$ into basic forms. The first property now follows from the definition of~$\sigma_{(M,\pairing)}$. The second assertion is an immediate consequence of the definition of the signature jump over~$\LR$.
\end{proof}

Next, we move on to the behavior of the signature function under Witt equivalence. Clearly~$\sigma_{(M,\pairing)}$ is not well-defined on Witt groups: for instance,~$\ee(2,1,\xi,\F)$ is metabolic (see Lemmas~\ref{lem:LRxi} and~\ref{lem:xiless1a}) but its signature function at~$\xi$ is equal to~$1$. More generally, this behavior of the signature function can be traced back to the following quantity which is not Witt invariant:
\begin{equation}\label{eq:sigma_loc}
  \sigma^{loc}_{(M,\pairing)}(\xi):=-\sum_{\substack{\eps=\pm 1\\n\textrm{ even}}} \eps \hodgep(n,\eps,\xi,\F).
\end{equation}
Since we saw in Theorem~\ref{thm:metabolic_signature} that the signature jumps~$\delta \sigma_{(M,\pairing)}$ are invariant under Witt equivalence, we deduce that the signature function defines a function on the Witt group for all but finitely many values of~$\xi$. In order to obtain a function which is Witt invariant on the whole circle, we use a well known construction which consists in taking averages at each point.

\begin{definition}\label{def:average_signature}
  The \emph{averaged signature}~$\sigma^{av}_{(M,\pairing)}(\xi)$ is defined as
  \begin{equation}\label{eq:average_signature}
    \sigma^{av}_{(M,\pairing)}(\xi)=\frac12\left(\lim_{\theta\to 0^+}\sigma_{(M,\pairing)}(e^{i\theta}\xi)+\lim_{\theta\to 0^-}\sigma_{(M,\pairing)}(e^{i\theta}\xi)\right).\end{equation}
\end{definition}

The next lemma relates the averaged signature to the signature jumps and proves its invariance under Witt equivalence. Briefly, averaging gets rid of the~$\sigma_{(M,\pairing)}^{\text{loc}}$ term in the definition of~$\sigma_{(M,\pairing)}$.

\begin{proposition}
  \label{prop:avsig}
  Let~$(M,\pairing)$ be a non-singular linking form. The averaged signature function of~$(M,\pairing)$ at~$\xi_1=e^{2\pi i\theta_1}$ can be described as follows:
  \begin{equation}  
    \label{eq:jump_av}
    \sigma^{av}_{(M,\pairing)}(\xi_1)=
    \begin{cases}
      \sigma_{(M,\pairing)}(\xi_1)-\sigma^{loc}_{(M,\pairing)}(\xi_1) & \text{if } \xi_1 \neq 1,   \\
      \sum \limits_{\xi\in S^1}\ds_{(M,\pairing)}(\xi) & \text{if } \xi_1=1.
    \end{cases}
  \end{equation}
  In particular, the averaged signature function~$\sigma^{av}_{(M,\pairing)}$ is invariant under Witt equivalence. In the real case, we have~$\sigma_{(M,\pairing)}^{\text{av}}(1)=\ds_{(M,\pairing)}(1)$.
\end{proposition}
\begin{proof}
  We first assume that~$\xi_1 \neq 1$. Let~$\theta_0< 1$ be small enough so that for all~$\theta\in[-\theta_0,\theta_0]\setminus\{0\}$ and all~$n,\eps$, one has 
 ~$\hodgep(n,\eps,e^{2\pi i \theta}\xi,\F)=0$. Note that such a~$\theta_0$ always exists, because~$\hodgep(n,\eps,\xi',\F)$ can only be non-zero only for finitely many~$\xi'$.
  Using successively the definition of the averaged signature function and of the signature function (and assuming that~$\theta\in (0,\theta_0)$ at each of these two steps), 
  we obtain
  \begin{align*}
    \sigma_{(M,\pairing)}^{av}(\xi_1) 
    &=\frac12(\sigma_{(M,\pairing)}(e^{2\pi i\theta}\xi)+\sigma_{(M,\pairing)}(e^{-2\pi i\theta}\xi)) \\
    &=\frac12\sum_{\tau\in(0,\theta_1+\theta_0)}2\ds_{(M,\pairing)}(e^{2\pi i\tau})+\frac12\sum_{\tau\in(0,\theta_1-\theta_0)}2\ds_{(M,\pairing)}(e^{2\pi i\tau})+\ds_{(M,\pairing)}(1).
  \end{align*}
  Next, grouping part of the second sum into the first, using once again our particular choice of~$\theta$ and recalling the definitions of 
 ~$\sigma_{(M,\pairing)}(\xi_1)$ and~$\sigma^{loc}_{(M,\pairing)}(\xi_1)$, we get the desired equation:
  \begin{align*}
    \sigma_{(M,\pairing)}^{av}(\xi_1) 
    &=\sum_{\tau\in(0,\theta_1-\theta_0)}2\ds_{(M,\pairing)}(e^{2\pi i\tau})+\frac12\sum_{\tau\in(\theta_1-\theta_0,\theta_1+\theta_0)}2\ds_{(M,\pairing)}(e^{2\pi i\tau})+\ds_{(M,\pairing)}(1)  \\
    &=\sum_{\tau\in(0,\theta_1)}2\ds_{(M,\pairing)}(e^{2\pi i\tau})+\ds_{(M,\pairing)}(\xi_1)+\ds_{(M,\pairing)}(1) \\
    &= \sigma_{(M,\pairing)}(\xi_1)-\sigma^{loc}_{(M,\pairing)}(\xi_1).
  \end{align*}
  Next, we assume that~$\xi_1=1$. Using the definition of the averaged signature function and the definition of the signature function, we immediately see that
  \begin{equation}
    \label{eq:AveragedAtOne}
    \sigma^{\text{av}}_{(M,\pairing)}(1)=\frac{1}{2}\left( \lim_{\theta\to 0^+} \sigma_{(M,\pairing)}(e^{2\pi i\theta})+\lim_{\theta\to 1^-} \sigma_{(M,\pairing)}(e^{2\pi i\theta}) \right)=\frac{1}{2}\left(\ds_{(M,\pairing)}(1)+\lim_{\theta\to 1^-} 
      \sigma_{(M,\pairing)}(e^{2\pi i\theta}) \right).
  \end{equation}
  It therefore only remains to deal with the~$\lim_{\theta\to 1^-} \sigma_{(M,\pairing)}(e^{2\pi i\theta})$ term. Since the signature function is obtained by summing the signature jumps along the circle, 
  we deduce that ~$\lim_{\theta\to 1^-} \sigma_{(M,\pairing)}(e^{2\pi i\theta})=\sum_{\tau\in(0,1)}2\ds_{(M,\pairing)}(e^{2\pi i \tau})+\ds_{(M,\pairing)}(1)$. Plugging this back into~(\ref{eq:AveragedAtOne}) immediately concludes the proof of~(\ref{eq:jump_av}).
  
  The Witt invariance of the averaged signature now follows from Theorem~\ref{thm:metabolic_signature}: the signature jumps are already known to be invariant under Witt equivalence. Finally, the last statement is immediate: in the real case, the definition of the signature jump implies that~$\sum_{\xi \in S^1 \setminus \lbrace 1 \rbrace} \delta \sigma_{(M,\pairing)}(\xi)=0$ and therefore the equation~$\sigma_{(M,\pairing)}^{\text{av}}(1)=\ds_{(M,\pairing)}(1)$ follows from~(\ref{eq:jump_av}).
\end{proof}
The statement of Theorem~\ref{thm:metabolic_signature} can now be rephrased in terms of the signature function.

\begin{corollary}\label{cor:metabolic_signature}
  Given a non-singular linking form~$(M,\pairing)$ over~$\LF$, the following assertions are equivalent:
  \begin{enumerate}
  \item  The linking form~$(M,\pairing)$ is metabolic.
  \item  The signature function~$\sigma_{(M,\pairing)}$ is zero for all but finitely many values of~$\xi$. 
  \item The averaged signature function~$\sigma^{av}_{(M,\pairing)}$ is zero.
  \end{enumerate}
\end{corollary}

\subsection{Signatures of representable forms}\label{sec:sigrep}
We study the signature function of representable linking forms. Namely, if~$(M,\pairing)$ is a linking form represented by a matrix~$A(t)$, then we will relate the signatures of (evaluations of)~$A(t)$ to the signatures of the pairing.
\medbreak

Given a Hermitian matrix~$A(t)$ over~$\LF$ and~$\xi\in S^1$, we write ~$\sign A(\xi)$ for the signature of the complex Hermitian matrix~$A(\xi)$. If~$A(t)$ represents a linking form~$(M,\pairing)$, then the signature~$\sign A(\xi)$ is not an invariant of
the linking form: for instance~$A(t) \oplus (1)$ also represents~$(M,\pairing)$ but~$\sign (A(t) \oplus (1))(\xi)=\sign A(\xi) +1$. On the other hand, the following result shows that~$\sign A(\xi)-\sign A(\xi')$ is an invariant of~$(M,\pairing)$ for any~$\xi,\xi' \in S^1$ (see also~\cite[Lemma 3.2]{BorodzikFriedl}).

\begin{proposition}\label{prop:signature_well_defined}
  Let~$(M,\pairing)$ be a non-singular linking form over~$\LF$. If~$(M,\pairing)$ is represented by~$A(t)$, then for any~$\xi,\xi'$ in~$S^1$, the difference~$\sign A(\xi)-\sign A(\xi')$ does not depend on~$A(t)$.
\end{proposition}

\begin{proof}
  We begin the proof by recalling
  a result which is implicit in~\cite{RanickiExact} and which is stated in~\cite[Proposition 2.2]{BorodzikFriedl2} (see also~\cite[Proposition 3.1]{BorodzikFriedl}).

  \begin{proposition}
    \label{prop:Ranicki}
    Let~$A$ and~$B$ be Hermitian matrices over~$\LF$ with non-zero determinant. The linking forms~$\pairing_A$ and~$\pairing_B$ are isometric if and only if~$A$ and~$B$ are related by a sequence of the three following moves:
    \begin{enumerate}
    \item Replace~$C$ by~$PC{P^\#}^T$, where~$P$ is a matrix over~$\LF$ with~$\det(P)$ a unit of~$\LF$. 
    \item Replace~$C$ by~$C \oplus D$, where~$D$ is a Hermitian matrix over~$\LF$ with~$\det(D)$ a unit.
    \item The inverse of~$(2)$.
    \end{enumerate}
  \end{proposition}
  From Proposition~\ref{prop:Ranicki}, we quickly conclude the proof of Proposition~\ref{prop:signature_well_defined} because the operation (1) does not change the signature of~$A(\xi)$ 
  and~$A(\xi')$, while the operation (2) changes both signatures by the same number, compare \cite[proof of Lemma 3.2]{BorodzikFriedl}.
\end{proof}

The goal of the next few results is to show how the signature function of a representable linking form can be computed from the (difference of) signatures of its representing matrices. The following proposition takes the first step in this process by studying signature jumps.

\begin{proposition}
  \label{prop:JumpIsJump}
  Let~$\xi\in S^1$. If a non-singular linking form~$(M,\pairing)$ is representable by a matrix~$A(t)$, then the following equation holds:
  \begin{equation}\label{eq:jumpisjump}
    \lim_{\theta\to 0^+} \sign A(e^{i \theta}\xi)-\lim_{\theta\to 0^-}\sign A(e^{i \theta}\xi)=2\ds_{(M,\pairing)}(\xi).
  \end{equation}
  Moreover, one also has 
  \begin{equation}\label{eq:jumpisjump_half}
    \sign A(\xi)-\lim_{\theta\to 0^-}\sign A(e^{i \theta}\xi)=\ds_{(M,\pairing)}(\xi)+\sigma^{loc}_{(M,\pairing)}(\xi).
  \end{equation}
\end{proposition}

We delay the proof of Proposition~\ref{prop:JumpIsJump} to  Section~\ref{sec:JumpIsJump} in order to describe its consequences. The first corollary shows that, up to the signature jump of~$(M,\pairing)$ at~$1$, one can compute the signature function of~$(M,\pairing)$ from the signature of a representing matrix.

\begin{corollary}\label{cor:sigissig}
  If~$(M,\pairing)$ is represented by a matrix~$A(t)$, then for any~$\xi_0\in S^1$ one has
  \begin{equation}\label{eq:sigissig}
    \sigma_{(M,\pairing)}(\xi_0)=\sign A(\xi_0)-\lim_{\theta\to 0^+} \sign A(e^{i\theta})+\ds_{(M,\pairing)}(1).
  \end{equation}
\end{corollary}
\begin{proof}
  First, suppose that~$\det A(\xi_0)\neq 0$. Let~$\theta_0 \in (0,1]$ be such that~$e^{2\pi i\theta_0}=\xi_0$ and let 
 ~$\xi_1,\ldots,\xi_n$ be the elements in~$S^1$ for which~$\det A(\xi_j)=0$ and~$\xi_j=e^{2\pi i\theta_j}$ for some~$\theta_j\in(0,\theta_0)$. 
  Assuming that~$\xi_1,\ldots,\xi_n$ are ordered by increasing arguments, we claim that both sides of~\eqref{eq:sigissig} are equal to
  \[2\sum_{j=1}^n \ds_{(M,\pairing)}(\xi_j)+\ds_{(M,\pairing)}(1).\] 
  For the left-hand side of \eqref{eq:sigissig}, this is the definition of the signature, and so we deal with the right-hand side.
  As the signature function~$\xi\mapsto \sign A(\xi)$ is constant on the subset of~$S^1$ on which~$A(t)$ is invertible, the definition of the~$\xi_j$ implies that
  \[\sign A(\xi_0)-\lim_{\theta\to 0^+}\sign A(e^{i\theta})=\sum_{j=1}^n \left( \lim_{\theta\to 0^+}\sign A(e^{i\theta}\xi_j)-\lim_{\theta\to 0^-}\sign A(e^{i\theta}\xi_j) \right).\]
  By the first point of Proposition~\ref{prop:JumpIsJump}, i.e. by applying~\eqref{eq:jumpisjump}~$n$ times, we obtain
  \[\sum_{j=1}^n \left( \lim_{\theta\to 0^+}\sign A(e^{i\theta}\xi_j)-\lim_{\theta\to 0^-}\sign A(e^{i\theta}\xi_j) \right)=2\sum_{j=1}^n\ds_{(M,\pairing)}(\xi_j).\]
  Hence
  \[\sign A(\xi_0)-\lim_{\theta\to 0^+}\sign A(e^{i\theta})=\sum_{j=1}^n2\ds_{(M,\pairing)}(\xi_j).\]
  This concludes the proof of the corollary in the case where~$\det A(\xi_0) \neq 0$.

  Next, we suppose that~$\det A(\xi_0)=0$. Choose~$\theta>0$ so that if~$0<|\theta'|<\theta$, then we have~$\det A(e^{i\theta'}\xi_0)\neq 0$. Using this condition, we apply the first part of the proof to~$\xi'=e^{i\theta'}$ and obtain
 ~$$\sigma_{(M,\pairing)}(\xi')=\sign A(\xi')-\lim_{\theta\to 0^+} \sign A(e^{i\theta})+\ds_{(M,\pairing)}(1).$$
  In order to deduce the corresponding equality for~$\xi_0$, note that since~$\xi'$ is close to~$\xi_0$, the definition of the signature function implies that~$\sigma_{(M,\pairing)}(\xi_0)-\sigma_{(M,\pairing)}(\xi')=\ds_{(M,\pairing)}(\xi_0)+\sigma^{loc}_{(M,\pairing)}(\xi_0)$. 
  Combining these two observations, we see that
 ~$$ \sigma_{(M,\pairing)}(\xi_0)=\ds_{(M,\pairing)}(\xi_0)+\sigma^{\text{loc}}_{(M,\pairing)}(\xi_0)+\sign A(\xi')-\lim_{\theta\to 0^+} \sign A(e^{i\theta})+\ds_{(M,\pairing)}(1).~$$
  Consequently, to conclude the proof, it suffices to show that~$\ds_{(M,\pairing)}(\xi_0)+\sigma^{\text{loc}}_{(M,\pairing)}(\xi_0)+\sign A(\xi')$ is equal to~$\sign A(\xi_0)$. Now this follows immediately from~\eqref{eq:jumpisjump_half} by noting that  since~$A(e^{i\theta})$ is non-singular for~$\theta\in[\theta',\theta_0)$, we have~$\sign A(\xi')=\lim_{\theta\to 0^{-}}\sign A(e^{i\theta}\xi_0)$.
\end{proof}

The second consequence of Proposition~\ref{prop:JumpIsJump} is the analogue of Corollary~\ref{cor:sigissig} for the averaged signature. Before stating this result, we introduce some notation. Namely, for~$\xi\in S^1$, we write
\[\sign^{av}A(\xi)=\frac12\left(\lim_{\theta\to 0^+}\sign A(e^{i\theta}\xi)+\lim_{\theta\to 0^-}\sign A(e^{i\theta}\xi)\right)\]
and refer to it as the \emph{averaged signature} of~$A(t)$ at~$\xi$. The averaged signature of~$(M,\pairing)$ can now be described in terms of the averaged signature of~$A(t)$, without the extra~$\delta \sigma_{(M,\pairing)}(1)$ term.

\begin{corollary}\label{cor:sigissig2}
  If~$(M,\pairing)$ is represented by a matrix~$A(t)$, then for any~$\xi \in S^1$ one has
  \begin{equation}\label{eq:sigissig2}
    \sigma^{av}_{(M,\pairing)}(\xi)
    =\sign^{av} A(\xi)-\sign^{av}A(1).
  \end{equation}
\end{corollary}

\begin{proof}
  Combining Corollary~\ref{cor:sigissig} with the definition of the averaged signature, we obtain that
  \begin{equation}\label{eq:sigthirdpart}
    \sigma^{av}_{(M,\pairing)}(\xi)=\sign^{av}A(\xi)-\lim_{\theta\to 0^+}\sign^{av} A(e^{i\theta})+\ds_{(M,\pairing)}(1).
  \end{equation}
  Consequently, to conclude the proof of the corollary, it only 
  remains to show that 
  \begin{equation}\label{eq:withlimit}\lim_{\theta\to 0^+}\sign^{av} A(e^{i\theta})-\ds_{(M,\pairing)}(1)=\sign^{av}A(1).\end{equation}
  Choose~$\theta_0>0$ in such a way that for any~$\theta\in(-\theta_0,\theta_0)$,~$\theta\neq 0$, we have~$\det A(e^{i\theta})\neq 0$. Assuming that~$\theta'$ and~$\theta''$ lie in~$(0,\theta_0)$, we obtain
  \begin{align*}
    \lim_{\theta\to 0^+}\sign A(e^{i\theta})&=\lim_{\theta\to 0^+}\sign^{av} A(e^{i\theta})=\sign A(e^{i\theta'})=\sign^{av}A(e^{i\theta'}),\\
    \lim_{\theta\to 0^-}\sign A(e^{i\theta})&=\lim_{\theta\to 0^-}\sign^{av} A(e^{i\theta})=\sign A(e^{-i\theta''})=\sign^{av}A(e^{-i\theta''}).
  \end{align*}
  This allows us to deduce that
  \begin{align*}
    \sign^{av} A(1)&=\frac12\left(\sign A(e^{i\theta'})+\sign A(e^{-i\theta''})\right),\\
    \ds_{(M,\pairing)}(1)&=\frac12\left(\sign A(e^{i\theta'})-\sign A(e^{-i\theta''})\right).
  \end{align*}
  Equation~\eqref{eq:withlimit} follows immediately, concluding the proof of the corollary.
\end{proof}

We can now prove the converse of Proposition~\ref{prop:inversejumps}

\begin{proposition}\label{prop:sumofjumps}
  If~$(M,\pairing)$ is a representable linking form over~$\LF$, then
  \begin{equation}\label{eq:zerojump}\sum_{\xi\in S^1}\ds_{(M,\pairing)}(\xi)=0.\end{equation}
\end{proposition}
\begin{proof}
  By Corollary~\ref{cor:sigissig2} applied to~$\xi=1$, we infer that~$\sigma^{av}_{(M,\pairing)}(1)=0$. Using Proposition~\ref{prop:avsig}, we also know that~$\sigma^{av}_{(M,\pairing)}(1)=\sum_{\xi \in S^1} \ds_{(M,\pairing)}(\xi)~$. Combining these two observations immediately leads to the desired result.
\end{proof}

Observe that Proposition~\ref{prop:sumofjumps} provides a second proof that~$\ee(2n+1,\eps,\xi,\C)$ is not representable, recall Proposition~\ref{prop:verystupidexample}. In fact, since Theorem~\ref{thm:metabolic_signature} ensures that the signature jumps are Witt invariants, Proposition~\ref{prop:sumofjumps} shows that this linking form is not even Witt equivalent to a representable one, recall Remark~\ref{rem:Litherland}.

Combining Proposition~\ref{prop:inversejumps} and Proposition~\ref{prop:sumofjumps}, we obtain the following result.
\begin{corollary}
  \label{cor:WittRepresentability}
  Over~$\C[t^{\pm 1}]$, metabolic forms are representable, representability is invariant under Witt equivalence and the following statements are equivalent:
  \begin{enumerate}
  \item A linking form is representable.
  \item A linking form is Witt equivalent to a representable one.
  \item The \emph{total signature jump}~(\ref{eq:zerojump}) of a linking form vanishes.
  \end{enumerate}
\end{corollary}
\begin{proof}
  Using Theorem~\ref{thm:metabolic_signature}, metabolic forms have vanishing total signature jump and therefore Proposition~\ref{prop:inversejumps} guarantees representability. The second assertion follows immediately, while the three equivalences are consequences of Propositions~\ref{prop:sumofjumps} and Proposition~\ref{prop:inversejumps}. 
\end{proof}

We conclude this subsection with two remarks. First observe that Proposition~\ref{cor:WittRepresentability} contrasts strongly with the real case in which all linking forms are representable. Secondly, using Remark~\ref{rem:Litherland}, Corollary~\ref{cor:WittRepresentability} can be understood as providing a computation of the image of the boundary map~$W(\C(t)) \to W(\C(t),\C[t^{\pm 1}])$.

\subsection{Proof of Proposition~\ref{prop:JumpIsJump}}\label{sec:JumpIsJump}
Suppose that~$A(t)$ is a matrix over~$\LF$ which represents the non-singular linking form~$(M,\pairing)$.
Our strategy is to start with a particular case to which the general case will later be reduced.

\emph{Case 1.} We first suppose that~$A(t)$ is a diagonal matrix with Laurent polynomials~$a_1(t),\ldots,a_n(t)$ on its diagonal. This assumption implies both that the module~$M$
splits as the direct sum of the cyclic submodules~$M_j=\LF/a_j(t)$ and that the restriction of~$\pairing$ to each of the~$M_j$ is given by
\begin{equation}\label{eq:lambdaj} 
  \pairing_j \colon M_j\times M_j\to\OF/\LF,\ (x,y)\mapsto\frac{xy^\#}{a_j(t)}.
\end{equation}
Therefore, when~$A(t)$ is diagonal, it is sufficient to check~\eqref{eq:jumpisjump} and \eqref{eq:jumpisjump_half} for each~$(M_j,\pairing_j)$. From now on, we fix a~$j$ and write~$n_j$ for the order of~$a_j$ at~$t=\xi$.
Depending on the sign of~$a_j(e^{i\theta}\xi)$ near~$\xi$, we consider four cases (here~$\theta$ is a small real number):
\begin{itemize}
\item \emph{Case (I):}~$a_j(e^{i\theta}\xi)$ changes sign from positive to negative as~$\theta$ goes from negative to positive;
\item \emph{Case (II):}~$a_j(e^{i\theta}\xi)$ changes sign from negative to positive as~$\theta$ goes from negative to positive;
\item \emph{Case (III):}~$a_j(e^{i\theta}\xi)$ is negative for~$\theta\neq 0$ and~$a_j(\xi)=0$;
\item \emph{Case (IV):}~$a_j(e^{i\theta}\xi)$ is positive for~$\theta\neq 0$ and~$a_j(\xi)=0$.
\end{itemize}
We calculate the jumps of the signatures and put them in a table:

\smallskip
\begin{center}
  \begin{tabular}
    {|c|c|c|c|c|}\hline
    Case & I & II & III & IV \\\hline
    parity of~$n_j$ & odd & odd & even & even \\\hline
    value of~$\eps_j$ &~$+1$ &~$-1$ &~$-1$ &~$+1$ \\ \hline
   ~$\ds_{(M,\pairing)}(\xi)$ &~$-1$ &~$1$ &~$0$ &~$0$ \\ \hline
   ~$\sigma_{(M,\pairing)}^{loc}$ &~$0$ &~$0$ &~$1$ &~$-1$ \\\hline
    l.h.s. of \eqref{eq:jumpisjump} &~$-2$ &~$2$ &~$0$ &~$0$ \\ \hline
    l.h.s. of \eqref{eq:jumpisjump_half} &~$-1$ &~$1$ &~$1$ &~$-1$ \\\hline
  \end{tabular}
\end{center}

\smallskip
The sign of~$\eps_j$ is calculated using Proposition~\ref{prop:cyclic_classif} above.
The left-hand sides of \eqref{eq:jumpisjump} and \eqref{eq:jumpisjump_half} are easily calculated using the~$1\times 1$ matrix~$A_j=(a_j)$.
From the table it immediately follows that \eqref{eq:jumpisjump} and \eqref{eq:jumpisjump_half} hold for~$1\times 1$ matrices, and hence, also in the case when the matrix~$A(t)$ is diagonal of arbitrary size.

\smallskip
\emph{Case 2.} We want to reduce the general case to the diagonal case.
In the real case, present~$(M,\pairing)$ as a direct sum~$(M',\pairing')\oplus(M_0,\pairing_0)$, where~$(M_0,\pairing_0)$ is a
direct sum of a finite number of copies of~$\ff(n,\xi,R)$ with~$\xi\in\{-1,1\}$ and~$n$ odd. According to Proposition~\ref{prop:diagonalreal}, there is a matrix~$A(t)=A'(t)\oplus A_0(t)$ representing~$(M,\pairing)$ such that~$A'$ represents~$(M',\pairing')$ and is diagonal, while~$A_0$
represents~$(M_0,\pairing_0)$ and is a finite number of copies of~$H$ from \eqref{eq:diagonal_H}.
Case~1 shows that the statement of Proposition~\ref{prop:JumpIsJump} holds for~$(M',\pairing')$. It is also satisfied for~$(M_0,\pairing_0)$ and~$A_0$, because the signature of~$A_0$ is constant zero and~$\ff(n,\xi,\R)$ does not contribute either to signature jumps or to~$\sigma^{loc}$.

From now on, we will assume that~$\F=\C$.
Since the proofs of \eqref{eq:jumpisjump} and of \eqref{eq:jumpisjump_half} are very similar, we focus on the proof
the former,
leaving the proof
the latter to the reader.

In order to avoid dealing with the delicate task of deciding whether~$A(t)$ is diagonalisable over~$\LC$, we instead pass to a local ring~$\OO_\xi$.
More precisely, recall from Subsection~\ref{sub:LocalizationDiagonalization} that~$\OO_\xi$ denotes the local ring
of analytic functions near~$\xi$ and that by Lemma~\ref{lem:analytic_change}, we can find size~$n$ matrices~$P(t)$ and~$B(t)$ over~$\OO_\xi$ such that~$B(t)=P(t)A(t){P(t)^\#}^T$ is diagonal for all~$t$ close to~$\xi$ and~$P(t)$ is invertible for all~$t$ close to~$\xi$.
In particular, the signature jump of~$B(t)$ at~$\xi$ is equal to the jump of the signature of~$A(t)$:
\begin{equation}
  \label{eq:dsignAdsignB}
  \lim_{\theta\to 0^+} \sgn A(e^{i \theta}\xi)-\lim_{\theta\to 0^-}\sgn A(e^{i \theta}\xi)=
  \lim_{\theta\to 0^+} \sgn B(e^{i \theta}\xi)-\lim_{\theta\to 0^-}\sgn B(e^{i \theta}\xi).
\end{equation}
We cannot immediately deduce \eqref{eq:jumpisjump} since~$B(t)$ might not represent~$(M,\pairing)$ over~$\LC$:~$B(t)$ is not even a matrix over~$\LC$.
On the other hand,~$B(t)$ does represent the linking form~$(\widehat{M},\widehat{\pairing})$ over~$\OO_\xi$, where
\[\wh{M}=\OO_\xi^n/B^T\OO_\xi^n\,\ \ \wh{\pairing}(x,y)=xB^{-1}y^\#\in \Omega_\xi/\OO_\xi.\]
Use~$b_1(t),\ldots,b_n(t)$ to denote the diagonal elements of~$B(t)$.
As~$B(t)$ is diagonal, the linking form~$(\widehat{M},\widehat{\pairing})$ splits as a direct sum of linking forms~$(\wh{M}_j,\widehat{\pairing}_j)$, where~$\wh{M}_j:=\OO_\xi/b_j(t)\OO_\xi$ and
\[ \widehat{\pairing}_j \colon \wh{M}_j\times \wh{M}_j\to\Omega_\xi/\OO_\xi,\ (x,y)\mapsto\frac{xy^\#}{b_j(t)}.\]
Since~$\widehat{M}_j$ is a cyclic~$\OO_\xi$-module, Proposition~\ref{prop:classify_cyclic} guarantees that~$(\wh{M}_j,\widehat{\pairing}_j)$ is isometric to the linking form~$\wh{\ee}(n_j,0,\xi,\epsilon_j,\C)$, where ~$n_j$ is the order of zero of~$b_j(t)$ at~$t=\xi$ (note that~$n_j$ might be zero, if~$b_j(\xi)\neq 0$) and~$\epsilon_j$ is described as follows:
\begin{itemize}
\item if~$n_j$ is even, then~$\epsilon_j$ is the sign of~$b_j(t)/((t-\xi)(t^{-1}-\ol{\xi}))^{n_j/2}$ at~$t=\xi$; 
\item if~$n_j$ is odd then if~$b_j(t)/((t-\xi)((t-\xi)(t^{-1}-\ol{\xi}))^{n_j/2})$ is~$\xi$-positive, we set~$\epsilon=1$, otherwise we set~$\eps=-1$.
\end{itemize}
Note that when~$n_j$ is odd, the first possibility corresponds precisely to the situation where~$\theta\mapsto b_j(e^{i\theta})$
changes sign from positive to negative, when~$\theta$ crosses~$\theta_0$.
Next, we define
\[\wh{\ds}_B(\xi):=\sum_{\substack{j=1,\ldots,n\\ n_i\textrm{ odd}}} \eps_j.\]
The same arguments as in Case~1 (that is, essentially the table) imply that 
\[2\wh{\ds}_B(\xi)= \lim_{\theta\to 0^+} \sgn B(e^{i \theta}\xi)-\lim_{\theta\to 0^-}\sgn B(e^{i \theta}\xi).\]
By~\eqref{eq:dsignAdsignB}, it only remains to show that~$\wh{\ds}_B(\xi)=\ds_{(M,\pairing)}(\xi)$.

Note that~$A(t)$ and~$B(t)$ represent the same pairing over~$\OO_\xi$.
That is,~$A(t)$ represents~$\bigoplus_{j=1}^n \wh{\ee}(n_i,\xi,\eps_i)$ over~$\OO_\xi$.
On the other hand,~$A(t)$ represents the pairing~$(M,\pairing)$ over~$\LC$.
By Proposition~\ref{prop:tensor}, we infer that the~$(t-\xi)$--primary part of~$(M,\pairing)$ is precisely~$\bigoplus_{i=1}^n \ee(n_i,\xi,\eps_i,\C)$.
Therefore, by definition of the signature jump, we obtain 
$$\ds_{(M,\pairing)}(\xi)=\sum_{\substack{j=1,\ldots,n\\ n_i\textrm{odd}}} \eps_j.$$
As this is the definition of~$\delta \widehat{\sigma}_B(\xi)$, Case 2 is concluded and so is the proof of the proposition.

\bibliographystyle{plain}
\def\MR#1{}
\bibliography{research}
\end{document}